\documentclass[journal]{IEEEtran}
\makeatletter

\let\proof\@undefined
\let\endproof\@undefined
\makeatother

\usepackage{epsfig,subfigure,graphicx}
\usepackage{amsmath}
\usepackage{amsthm}
\usepackage{amssymb}

\newtheorem{thm}{Theorem}
\theoremstyle{definition}
\theoremstyle{remark}

\theoremstyle{plain}

\theoremstyle{remark}

\theoremstyle{plain}
\newtheorem{cor}{Corollary}
\theoremstyle{plain}
\newtheorem{lem}{Lemma}

\newcommand{\oper}[1]{\mathcal{#1}}

\newcommand{\norm}[1]{\left\|#1\right\|}
\newcommand{\inner}[1]{\left<#1\right>}
\newcommand{\R}{\mathbb{R}}

\newcommand{\C}{\mathbb{C}}
\newcommand{\E}{\mathbb{E}}
\newcommand{\Prob}{\mathbb{P}}

\DeclareMathOperator*{\sgn}{sgn}
\DeclareMathOperator*{\Var}{Var}

\DeclareMathOperator{\diag}{diag}

\allowdisplaybreaks

\begin{document}

\title{Robust Lasso with missing and grossly corrupted observations}

\author{Nam H. Nguyen 
        and Trac D. Tran,~\IEEEmembership{Senior Member,~IEEE}
\thanks{This work has been partially supported by the National Science Foundation (NSF) under Grants CCF-1117545 and CCF-0728893; the Army Research Office (ARO) under Grant 58110-MA-II and Grant
60219-MA; and the Office of Naval Research (ONR) under Grant N102-183-0208.}
\thanks{Nam H. Nguyen and Trac D. Tran are with the Department of Electrical and Computer Engineering, the Johns Hopkins University, Baltimore, MD, 21218 USA (email: {nam, trac}@jhu.edu ).}
\thanks{Partial of this work is presented at NIPS 2011 conference in Granda, Spain, December 2011.}
}


\date{}
\maketitle

\begin{abstract}
This paper studies the problem of accurately recovering a sparse vector $\beta^{\star}$ from highly corrupted linear measurements $y = X \beta^{\star} + e^{\star} + w$ where $e^{\star}$ is a sparse error vector whose nonzero entries may be unbounded and $w$ is a bounded noise. We propose a so-called extended Lasso optimization which takes into consideration sparse prior information of both $\beta^{\star}$ and $e^{\star}$. Our first result shows that the extended Lasso can faithfully recover both the regression as well as the corruption vector. Our analysis relies on the notion of extended restricted eigenvalue for the design matrix $X$. Our second set of results applies to a general class of Gaussian design matrix $X$ with i.i.d rows $\oper N(0, \Sigma)$, for which we can establish a surprising result: the extended Lasso can recover exact signed supports of both $\beta^{\star}$ and $e^{\star}$ from only $\Omega(k \log p \log n)$ observations, even when the fraction of corruption is arbitrarily close to one. Our analysis also shows that this amount of observations required to achieve exact signed support is indeed optimal.

\end{abstract}

\section{Introduction}

One of the central problems in statistics is the problem of linear regression in which the goal is to accurately estimate the regression vector $\beta^{\star} \in \R^p$ from the noisy observations
\begin{equation}
\label{eqt::linear observations}
y = X \beta^{\star} + w,
\end{equation}
where $X \in \R^{n \times p}$ is the measurement or design matrix, and $w \in \R^n$ is the stochastic observation vector noise. A particular situation recently attracted much attention from the research community concerns with the model in which the number of regression variables $p$ is larger than the number of observations $n$ ($p \geq n$). In such circumstances, without imposing some additional assumptions for this model, it is obvious that the problem is ill-posed, and thus the linear regression is not consistent. Accordingly, there have been various lines of work on high dimensional inference based on imposing different types of structure constraints such as sparsity and group sparsity (e.g. \cite{Tibshirani_Lasso_1996_J}, \cite{CT_Dantzig_2007_J}, \cite{YL_2006_J}, \cite{ZRY_2009_J}, \cite{ZY_Lasso_2006_J}, \cite{MY_Lasso_2009_J}, \cite{MB_Lasso_2008_J}, \cite{Wainwright_Lasso_2009_J}, \cite{CP_Lasso_2009_J}, \cite{BRT_Lasso_2009_J}, \cite{Zhang_2009_J}, \cite{BTW_2007_J}, \cite{Bunea_2008}, \cite{OWJ_2011_J}, \cite{HZ_2010_J}). Among them, the most popular model focused on sparsity assumption of the regression vector. To estimate $\beta$, a standard method, namely Lasso \cite{Tibshirani_Lasso_1996_J}, was proposed to use $l_1$-penalty as a surrogate function to enforce sparsity constraint.
\begin{equation}
\label{opt::Lasso}
\min_{\beta} \frac{1}{2n} \norm{y - X \beta}_2^2 + \lambda \norm{\beta}_1,
\end{equation}
where $\lambda$ is the positive regularization parameter and the $\ell_1$-norm of the regression vector is $\norm{\beta}_1$, defined as $\norm{\beta}_1 = \sum_{i=1}^p |\beta_i|$.

Within the past few years, there has been numerous studies to understand the $\ell_1$-regularization aspect of sparse regression models (e.g. \cite{ZY_Lasso_2006_J}, \cite{MY_Lasso_2009_J}, \cite{MB_Lasso_2008_J}, \cite{Wainwright_Lasso_2009_J}, \cite{CP_Lasso_2009_J}, \cite{BRT_Lasso_2009_J}, \cite{Zhang_2009_J}). These works are mainly characterized by the type of the loss functions considered. For instance, authors \cite{CP_Lasso_2009_J}, \cite{Zhang_2009_J} seek to obtain a regression estimate $\widehat{\beta}$ that delivers small prediction error while others \cite{BRT_Lasso_2009_J}, \cite{MY_Lasso_2009_J} \cite{Zhang_2009_J} seek to produce a regressor with minimal parameter estimation error, which is measured by the $\ell_2$-norm of $(\widehat{\beta}-\beta^{\star})$. Another line of work (e.g. \cite{Tropp_Relax_2006_J}, \cite{ZY_Lasso_2006_J}, \cite{Wainwright_Lasso_2009_J}) considers the variable selection in which the goal is to obtain an estimate that correctly identifies the support of the true regression vector. To achieve low prediction or parameter estimation loss, it is now well known that it is both sufficient and necessary to impose certain lower bounds on the smallest singular values of the design matrix (e.g. \cite{MB_Lasso_2008_J}, \cite{BRT_Lasso_2009_J}), while the notion of small mutual coherence for the design matrix (e.g. \cite{CP_Lasso_2009_J}, \cite{ZY_Lasso_2006_J}, \cite{Wainwright_Lasso_2009_J}) is required to achieve accurate variable selection.



We notice that all previous work relies on the assumption that the observation noise has bounded energy. Without this assumption, it is very likely that the estimated regressor is either not reliable or we fail to identify the correct support. With this observation in mind, in this paper, we extend the linear model (\ref{eqt::linear observations}) by considering the noise with unbounded energy. It is clear that if all entries of $y$ are corrupted by large errors, then it is impossible to faithfully recover the regression vector $\beta^{\star}$. However, in many practical applications such as face recognition, acoustic recognition and dense sensor network, only a portion of the observation vector is contaminated by gross error. Formally, we have the mathematical model
\begin{equation}
\label{eqt::extend linear observations}
y = X \beta^{\star} + e^{\star} + w,
\end{equation}
where $e^{\star} \in \R^n$ is the sparse error whose locations of nonzero entries are unknown and whose magnitudes can be arbitrarily large whereas $w$ is the conventional noise vector with bounded entries. In this paper, we assume that $w$ has a multivariate Gaussian $\oper N(0,\sigma^2 I_{n\times n})$ distribution. This model also includes as a special case the missing data problem in which all the entries of $y$ is not fully observed, but some are missing. This problem is particularly important in computer vision and biology applications. If some entries of $y$ are missing, the nonzero entries of $e^{\star}$ whose locations are associated with the missing entries of the observation vector $y$ have the same values as entries of $y$ but with reverse polarity.

The problems of faithfully recovering data under gross error has gained increasing attentions recently with many interesting practical applications (e.g. \cite{WYGSM_Face_2009_J}, \cite{EV_2009_C}, \cite{LDB_corruption_2009_C}) as well as theoretical consideration (e.g. \cite{WM_denseError_2010_J}, \cite{LWW_grossError_2010_C}, \cite{NT_GrossError_2010_J}, \cite{Li_2011_J}). Another recent line of research on recovering the data from grossly corrupted measurements has been also studied in the context of robust principal component analysis (RPCA) (e.g. \cite{CLMW_RobustPCA_2009_J}, \cite{XCS_RPCA_2010_C}, \cite{ANW_RPCA_2011_C}). Let us consider several examples as illustrations.
\begin{itemize}
    \item \textit{Face recognition.} The model (\ref{eqt::extend linear observations}) has been proposed by Wright \textit{et al.} \cite{WYGSM_Face_2009_J} in the context of face recognition. In this problem, a face test sample $y$ is assumed to be represented as a linear combination of training faces in the dictionary $X$. Hence, $y = X \beta$ where $\beta$ is the coefficient vector used for classification. However, it is often the case that the testing face of interest is occluded by unwanted objects such as glasses, hats, scarfs, etc. These occlusions, which occupy a portion of the test face, can be considered as the sparse error $e^{\star}$ in the model (\ref{eqt::extend linear observations}).



    \item \textit{Subspace clustering.} An important problem in high-dimensional data analysis is to cluster the data points into multiple subspaces. A recent work of Elhamifar and Vidal \cite{EV_2009_C} show that this problem can be solved by expressing each data point as a sparse linear combination of all other data points. Coefficient vectors recovered from solving the Lasso problems are then employed for clustering. If the data points are represented as a matrix $X$, then we wish to find a sparse coefficient matrix $B$ such that $X = X B$ and $\diag(B) = 0$. When the data is missing or contaminated by outliers, the authors formulate the problem as $X = XB + E$ and minimize a sum of two $\ell_1$-norms with respect to both $B$ and $E$ \cite{EV_2009_C}.




    \item \textit{Sparse graphical model estimation.} Given a random vector $x \in \R^p$ with unkown covariance matrix $\Sigma$, the goal is to estimate $\Sigma$ or its precision matrix $\Omega = \Sigma^{-1}$ from $n$ independent copies of $x$: $x_1, ..., x_n \in \R^p$. Assuming that the matrix $\Omega$ is sparse, Meinshausen and B$\ddot{\text{u}}$hlmann \cite{MB_Lasso_2008_J} propose to solve the following Lasso problem
        $$
        \min_B \frac{1}{2n} \norm{X - XB}_F^2 + \lambda \norm{B}_1 \quad \text{s.t.  }\diag(B) = 0,
        $$
        where $X = [x_1^T,...,x_n^T]$. The precision matrix $\Omega$ can be estimated via the coefficient matrix $B$. When the data $X$ is partially observed/missing, a more robust method is to take into account the sparsity assumption and minimize
        $$
        \min_{B,E} \frac{1}{2n} \norm{X - XB - E}_F^2 + \lambda_b \norm{B}_1 + \lambda_e \norm{E}_1
        $$
        subject to $\diag(B) = 0$, where $E$ represents partially missing information. Though this problem is quite different from the aforementioned subspace clustering problem, the technical approach is considerably similar.

   	\item \textit{Sensor network.} In this model, a network of sensors collect measurements of a signal $\beta^{\star}$ independently by simply projecting $\beta^{\star}$ onto the row vectors of a sensing matrix $X$, $y_i = \inner{X_i, \beta^{\star}}$ \cite{HBRN_2008_J}. The measurements $y_i$ are then sent to the central hub for analysis. However, it is highly likely that a small percentage of sensors might fail to send the measurements correctly and sometimes even report totally irrelevant measurements. Therefore, it is more appropriate to employ the observation model in (\ref{eqt::extend linear observations}) than the model in (\ref{eqt::linear observations}).



\end{itemize}


It is worth noticing that in the aforementioned applications, $e^{\star}$ always plays the role as the sparse (undesired) error. However, in other applications, $e^{\star}$ might actually contain meaningful information, and thus necessary to be recovered. An example of this kind of problem is signal separation, in which $\beta^{\star}$ and $e^{\star}$ are considered as two distinct signal components (e.g. video or audio). Furthermore, in applications such as classification and clustering, the assumption that the test sample $y$ is a linear combination of a few training samples in the dictionary (playing the role of the design matrix) $X$ might be violated. The sparse component $e^{\star}$ can thus be seen as the compensation for the linear regression model mismatch.



Given the observation model (\ref{eqt::linear observations}) and the sparsity assumptions on both regression vector $\beta^{\star}$ and error $e^{\star}$, we propose the following convex minimization to estimate the unknown regression vector $\beta^{\star}$ as well as the error vector $e^{\star}$.
\begin{equation}
\label{opt::extended lasso}
\min_{\beta, e} \frac{1}{2n} \norm{y - X \beta - e}_2^2 + \lambda_{n,\beta} \norm{\beta}_1 + \lambda_{n,e} \norm{e}_1,
\end{equation}
where $\lambda_{n,\beta}$ and $\lambda_{n,e}$ are positive regularization parameters. This optimization, which we call \textit{extended Lasso}, can be seen as a generalization of the Lasso program. Indeed, by setting $\lambda_{n,e} = 0$, (\ref{opt::lasso with sparse noise}) returns to the standard Lasso. The additional regularization associated with the error $e$ encourages sparsity of the reconstructed vector, where the penalty parameter $\lambda_{n,e}$ controls its sparsity level. In this paper, we focus on the following questions: what are necessary and sufficient conditions for the ambient dimension $p$, the number of observations $n$, the sparsity index $k$ of the regression $\beta^{\star}$ and the fraction of corruption in $e^{\star}$ so that (i) the extended Lasso is able (or unable) to recover the exact support sets of both $\beta^{\star}$ and $e^{\star}$? (ii) the extended Lasso is able to recover $\beta^{\star}$ and $e^{\star}$ with small prediction error and parameter error? We are particularly interested in understanding the asymptotic situation where the the fraction of error gets arbitrarily close to $100\%$.

In this paper, we assume normalization of the design matrix $X$. Specifically, we assume the $\ell_2$-norm of columns of the matrix $X$ are $\Theta(\sqrt{n})$. Moreover, without loss of generality, we use the following observation model in replacement for the model in (\ref{eqt::extend linear observations})
\begin{equation}
\label{eqt::extend linear observations 2}
y = X \beta^{\star} + \sqrt{n} e^{\star} + w.
\end{equation}
As we can see, columns of both the design matrix $X$ and the matrix $\sqrt{n}I_{n\times n}$ has the same scale. Thus, this model's change only helps our results in the next sections to be more interpretable. The optimization (\ref{opt::extended lasso}) is now converted to the following problem
\begin{equation}
\label{opt::lasso with sparse noise}
\min_{\beta, e} \frac{1}{2n} \norm{y - X \beta - \sqrt{n} e}_2^2 + \lambda_{n,\beta} \norm{\beta}_1 + \lambda_{n,e} \norm{e}_1,
\end{equation}


\textit{Previous work.} The problem of recovering the estimation vector $\beta^{\star}$ and error $e^{\star}$ is originally proposed by Wright \textit{et al.} in the appealing paper \cite{WYGSM_Face_2009_J} and analyzed by Wright and Ma \cite{WM_denseError_2010_J}. In the absence of the stochastic noise $w$ in the observation model (\ref{eqt::extend linear observations}), the authors propose to estimate ($\beta^{\star}, e^{\star}$) by solving the following linear program
\begin{equation}
\label{opt::extended L1}
\min_{\beta, e} \norm{\beta}_1 + \norm{e}_1 \quad\text{s.t.} \quad y = X \beta + \sqrt{n} e.
\end{equation}


From a different viewpoint, in the intriguing paper \cite{LMJ_2012_J}, Lee \textit{et al.} study a general loss function model. To obtain more flexibility in controlling the undesirable influence of the model, they introduce a case-specific parameter vector $e \in \R^n$ for the observation vectors and modify the optimization to take into account this parameter. Interestingly, the model turns out to be coincident with (\ref{opt::lasso with sparse noise}) when applying to the linear regression problem with Lasso penalty. Extensive simulations have shown that the model (\ref{opt::lasso with sparse noise}) is considerably robust to noise. However, no theoretical analysis is provided in the paper.

In another direction, the problem of robust Lasso under corrupted observations is also carefully investigated by Wang \textit{et al.} \cite{WLJ_2007_J}. In this appealing paper, instead of using the quadratic loss function as in Lasso, the authors propose to employ LAD-Lasso criterion:
\begin{equation}
\label{opt::LAD}
\min_{\beta} \norm{y - X\beta}_1 + \sum_{j=1}^p \lambda_j |x_j|.
\end{equation}
This optimization combines the LAD criterion and Lasso penalty, where the first term is designed to be robust to outliers and the second term again promotes the sparse representation of the estimator. However, due to the lack of the quadratic loss that enforce the estimation to be consistence with the observation in $\ell_2$-norm sense, this optimization might not guarantee to deliver a solution that satisfies small prediction error.

On the theoretical side, the result of \cite{WM_denseError_2010_J} is asymptotic in nature. The analysis reveals that for a class of Gaussian design matrix with i.i.d entries, the optimization (\ref{opt::extended L1}) can recover $(\beta^{\star}, e^{\star})$ precisely with high probability even when the fraction of corruption is arbitrarily close to one. However, the result only holds under rather stringent conditions. In particularly, the authors require the number of observations $n$ grow proportionally with the ambient dimension $p$, and the sparsity index $k$ is a very small portion of $n$. These conditions is of course far from the optimal bound in compressed sensing (CS) and statistics literature (recall $k \leq O(n/\log p)$ is sufficient in conventional analysis (e.g. \cite{CRT_CS_2004_J}, \cite{Wainwright_Lasso_2009_J}).

Another line of work has also focused on the optimization (\ref{opt::extended L1}). In both Laska \textit{et al.} \cite{LDB_corruption_2009_C} and Li \textit{et al.}, \cite{LWW_grossError_2010_C}, the authors establish that for Gaussian design matrix $X$, if $n \geq C (k+s) \log p$ where $s$ is the sparsity level of $e^{\star}$, then the recovery is exact. This follows from the fact that the combination matrix $[X, \text{ } I]$ obeys the restricted isometry property, a well-known property in compressed sensing used to guarantee exact recovery of sparse vectors via $\ell_1$-minimization. These results, however, do not allow the fraction of corruption to come close to unity. Also related to our paper is recent work by Studer \textit{et al.}, \cite{SKPB_2011_C} \cite{SB_2011_J} in which the authors establish different results for deterministic design matrix.

Among the previous work, the most closely related to our current paper are recent results by Li \cite{Li_2011_J} and Nguyen \textit{et al.} \cite{NT_GrossError_2010_J} in which a positive regularization parameter $\lambda$ is employed to control the sparsity of $e^{\star}$. Using different methods, both sets of authors show that as $\lambda$ is deterministically selected to be $1/\sqrt{\log p}$ and $X$ is a sub-orthogonal matrix, whose columns are selected uniformly at random from columns of an orthogonal matrix, then the solution of following optimization (\ref{opt::extended L1 with lambda}) is exact even a constant fraction of observation is corrupted. Moreover, \cite{Li_2011_J} establishes a similar result with Gaussian design matrix in which the number of observations is only on the order of $k \log p$ $-$ a level that is known to be optimal in both CS and statistics community.
\begin{equation}
\label{opt::extended L1 with lambda}
\min_{\beta, e} \norm{\beta}_1 + \lambda \norm{e}_1 \quad\text{s.t.} \quad y = X \beta + \sqrt{n} e.
\end{equation}

\textit{Our contribution.} This paper considers a general setting in which the observations are contaminated by both sparse and dense errors. We allow the corruptions to linearly grow with the number of observations and have arbitrarily large magnitudes. We establish a general scaling of the quadruplet $(n,p,k,s)$ such that the proposed extended Lasso stably recovers both the regression and the corruption vector. Of particular interest to us are the answer to the following questions:
\begin{enumerate}
    \item [(a)] First, under what scalings of $(n,p,k,s)$ does the extended Lasso obtain the unique solution with small estimation error?
    \item [(b)] Second, under what scalings of $(n,p,k)$ does the extended Lasso obtain the exact signed support recovery even when almost all observations are corrupted?
    \item [(c)] Third, under what scalings of $(n,p,k,s)$ that no solution of the extended Lasso specifying the correct signed support exists?
\end{enumerate}

To answer for the first question, we introduce a notion of \textit{extended restricted eigenvalue} for a matrix $[X, \text{ } I]$ where $I$ is the identity matrix. We show that this property is satisfied for a general class of random Gaussian design matrices. The answers to the last two questions requires stricter conditions on the design matrix. In particular, for random Gaussian design matrix with i.i.d rows $\oper N(0,\Sigma)$, we rely on two standard assumptions: invertibility and mutual incoherence. Our analysis in this setting is relied on the elegant technique introduced by Wainwright \cite{Wainwright_Lasso_2009_J}.

If we denote $Z = [X,\text{ } I]$ where $I$ is an identity matrix and $\overline{\beta} = [\beta^{\star^T}, \text{ } e^{\star^T}]^T$, then the observation vector $y$ is reformulated as $y = Z \overline{\beta} + w$, which is the same as the standard Lasso model. However, previous results (e.g. \cite{BRT_Lasso_2009_J}, \cite{Wainwright_Lasso_2009_J}) applying to random Gaussian design matrix are irrelevant to this setting since $Z$ no longer behaves like a Gaussian matrix. To establish the theoretical analysis, we need a deeper study on the interaction between the Gaussian and identity matrices. By exploiting the fact that the matrix $Z$ consists of two components where one has a special structure, our analysis reveals an interesting phenomenon: extended Lasso can accurately recover both the regressor $\beta^{\star}$ and the corruption $e^{\star}$ even when the fraction of corruption is up to $100\%$. We measure the recoverability of these variables under two criterions: parameter accuracy and feature selection accuracy. Moreover, our analysis can be extended to the situation in which the identity matrix can be replaced by a tight frame $D$ as well as extended to other models such as group Lasso or matrix Lasso with sparse error.






\textit{Notation.} We summarize here some standard notation employed throughout the paper. We reserve $T$ and $S$ as the sparse support of $\beta^{\star}$ and $e^{\star}$, respectively. Given a design matrix $X \in \R^{n \times p}$ and subsets $S$ and $T$, we use $X_{ST}$ to denote the $|S| \times |T|$ submatrix obtained by extracting those rows indexed by S and columns indexed by $T$. For a vector $h \in \R^p$, we use the conventional notations for $\ell_1$- and $\ell_2$-norm of $h$ as $\norm{h}_1 = \sum_{i=1}^p |h_i|$ and $\norm{h}_2 = (\sum_{i=1}^p h_i^2)^{1/2}$, respectively. For a matrix $X \in \R^{n \times p}$, we denote $\norm{X}$ and $\norm{X}_{\infty}$ as the operator norms. In particular, $\norm{X}$ is denoted as the spectral norm and $\norm{X}_{\infty}$ as the $\ell_{\infty}/\ell_{\infty}$ operator norm: $\norm{X}_{\infty} = \max_i \sum_{j=1}^p |x_{ij}|$.

We use the notation $C_1, C_2, c_1, c_2,$ etc., to refer to positive constants, whose value may change from line to line. Given two functions $f$ and $g$, the notation $f(n) = \oper O(g(n))$ means that there exists a constant $c < + \infty$ such that $f(n) \leq c g(n)$; the notation $f(n) = \Omega (g(n))$ means that $f(n) \geq c g(n)$ and the notation $f(n) = \Theta(g(n))$ means that $f(n) = \oper O (g(n))$ and $f(n) = \Omega(g(n))$. The symbol $f(n) = o(g(n))$ indicates that $f(n)/g(n) \rightarrow 0$.

\textit{Organization.} The remainder of this paper is structured as follows. Section \ref{sec::main results} provides the main results, detailed discussions and their consequences. Section \ref{sec::simulations} performs extensive experiments to validate theoretical results presented in the previous section. Section \ref{sec::proof 1} provides analysis of the estimation error, whereas Sections \ref{sec::proof 2} and \ref{sec::proof 3} deliver proofs of the necessary and sufficient conditions for the exact signed support recovery. Several technical aspects of these proofs and some well-known concentration inequalities are presented in the Appendix. We conclude the paper in Section \ref{sec::conclusion} with more discussion.

\section{Main results}
\label{sec::main results}

In this section, we provide precise statements for the main results of this paper. In the first sub-section, we establish the parameter estimation and provide a deterministic result which is based on the notion of extended restricted eigenvalue. We further show that the random Gaussian design matrix satisfies this property with high probability. The next sub-section considers feature estimation. We establish conditions for the design matrix such that the solution of the extended Lasso has the exact signed supports.

\subsection{Parameter estimation}

As in conventional Lasso, to obtain a low parameter estimation bound, it is necessary to impose conditions on the design matrix $X$. In this paper, we introduce the notion of \textit{extended restricted eigenvalue} (extended RE) condition. Let $\C$ be a restricted set, we say that the matrix $X$ satisfies the extended RE assumption over the set $\C$ if there exists some $\kappa_l > 0$ such that
\begin{equation}
\label{eqt::extended RE}
\frac{1}{\sqrt{n}} \norm{Xh + \sqrt{n} f}_2 \geq \kappa_l (\norm{h}_2 + \norm{f}_2) \quad \text{for all } (h, f) \in \C,
\end{equation}

\noindent where the restricted set $\C$ of interest is defined with $\lambda := \frac{\lambda_{n,e}}{\lambda_{n,\beta}}$ as follows
\begin{multline}
\label{eqt::set C}
\C := \{ (h, f) \in \R^p \times \R^n \text{ } | \\
\text{ } \norm{h_{T^c}}_1  + \lambda \norm{f_{S^c}}_1 \leq 3 \norm{h_T}_1 + 3\lambda \norm{f_S}_1 \}.
\end{multline}

This assumption is a natural extension of the restricted eigenvalue condition and restricted strong convexity considered in \cite{BRT_Lasso_2009_J} ,\cite{RWY_RE_2010_J} and \cite{NRWY_2010_J}. In the absence of a vector $f$ in the equation (\ref{eqt::extended RE}) and in the set $\C$, this condition returns to the restricted eigenvalue defined in \cite{BRT_Lasso_2009_J}. As discussed in more detail in \cite{BRT_Lasso_2009_J}  and \cite{GB_Lasso_2009_J}, restricted eigenvalue is among the weakest assumption on the design matrix such that the solution of the Lasso is consistent.


With this assumption at hand, we now state the first theorem
\begin{thm}
\label{thm::parameter estimation}
Consider the optimal solution $(\widehat{\beta}, \widehat{e})$ to the optimization problem (\ref{opt::lasso with sparse noise}) with regularization parameters chosen as
\begin{equation}
\lambda_{n,\beta} = \frac{2}{\gamma} \frac{\norm{X^* w}_{\infty}}{n} \quad \text{and} \quad \lambda_{n,e} = \frac{2 \norm{w}_{\infty}}{\sqrt{n}},
\end{equation}
where $\gamma \in (0,1]$. Assuming that the design matrix $X$ obeys the extended RE, then the error set $(h,f) = (\widehat{\beta} - \beta^{\star}, \widehat{e}-e^{\star})$ is bounded by
\begin{equation}
\label{inq::L2 error bound}
\norm{h}_2 + \norm{f}_2 \leq 3 \kappa^{-2}_l \left( \lambda_{n,\beta} \sqrt{k} + \lambda_{n,e} \sqrt{s} \right).
\end{equation}
\end{thm}

There are several interesting observations from this theorem

1) The error bound naturally split into two components related to the sparsity indices of $\beta^{\star}$ and $e^{\star}$. In addition, the error bound contains three quantity: the sparsity indices, regularization parameters, and the extended RE constant. If the terms related to the corruption $e^{\star}$ are omitted, then we obtain similar parameter estimation bound as in the standard Lasso (e.g. \cite{BRT_Lasso_2009_J}, \cite{NRWY_2010_J}).

2) The choice of regularization parameters $\lambda_{n,\beta}$ and $\lambda_{n,e}$ can be made explicitly: assuming $w$ is a Gaussian random vector whose entries are $\oper N(0,\sigma^2)$ and the design matrix has $\sqrt{n}$-normed columns, it is clear that with high probability, $\frac{1}{n}\norm{X^* w}_{\infty} \leq 2 \sqrt{\frac{\sigma^2 \log p}{n}}$ and $\frac{1}{\sqrt{n}}\norm{w}_{\infty} \leq 2 \sqrt{\frac{\sigma^2 \log n}{n}}$. Thus, it is sufficient to select $\lambda_{n,\beta} \geq \frac{4}{\gamma} \sqrt{\frac{\sigma^2 \log p}{n}}$ and $\lambda_{n,e} \geq 4 \sqrt{\frac{\sigma^2 \log n}{n}}$.

3) At the first glance, the parameter $\gamma$ does not seem to have any meaningful interpretation and the setting $\gamma = 1$ seems to be the best selection due to the smallest estimation error it can produce. However, this parameter actually controls the sparsity level of the regression vector with respect to the fraction of corruption. This relation is enforced via the restricted set $\C$.

In the following lemma, we show that the extended RE condition actually exists for a large class of random Gaussian design matrix whose rows are i.i.d zero mean with covariance $\Sigma$. Before stating the lemma, let us define some quantities operating on the covariance matrix $\Sigma$: $C_{\min} := \lambda_{\min} (\Sigma)$ is the smallest eigenvalue of $\Sigma$; $C_{\max} := \lambda_{\max} (\Sigma)$ is the largest eigenvalue of $\Sigma$; and $\xi(\Sigma) := \max_{i} \Sigma_{ii}$ is the maximal entry on the diagonal of the matrix $\Sigma$.

\begin{lem}
\label{lem::extended RE with Gaussian matrix}
Consider the random Gaussian design matrix whose rows are i.i.d $\oper N(0, \Sigma)$ and assume $C_{\max} \xi(\Sigma) = \Theta(1)$. Select
\begin{equation}
\label{eqt::lambda_n}
\lambda := \frac{\gamma}{\sqrt{\xi(\Sigma)}} \sqrt{\frac{\log n}{\log p}}.
\end{equation}
Then with probability greater than $1 - c_1 \exp(-c_2 n)$, the matrix $X$ satisfies the extended RE with parameter $\kappa_l = \frac{1}{4\sqrt{2}}$, provided that $ n \geq C \frac{\xi (\Sigma)}{C_{\min}} k \log p$ and $s \leq \min \left\{ C_1 \frac{n}{\gamma^2 \log n}, C_2 n \right\}$ for some small constants $C_1$, $C_2$.
\end{lem}

We would like to offer a few remarks:

1) The choice of parameter $\lambda$ is nothing special here. When the design matrix is Gaussian and independent with the Gaussian stochastic noise $w$, we can easily show that $\frac{1}{n}\norm{X^* w}_{\infty} \leq 2 \sqrt{\xi(\Sigma) \delta^2 \log p}$ with probability at least $1 - 2 \exp(-\log p)$. Therefore, the selection of $\lambda$ follows from Theorem \ref{thm::parameter estimation}.

2) The proof of this lemma, shown in the Appendix, boils down to controling two terms
\begin{itemize}
    \item \textit{Restricted eigenvalue with $X$.}
    $$
    \frac{1}{n}\norm{X h}^2_2 + \norm{f}_2^2 \geq \kappa_r (\norm{h}^2_2 + \norm{f}^2_2) \quad \text{for all  } (h,f) \in \C.
    $$
    \item \textit{Mutual incoherence.} The column space of the matrix $X$ is incoherent with the column space of the identity matrix. That is, there exists some $\kappa_m > 0$ such that
    $$
    \frac{1}{\sqrt{n}} |\inner{Xh, f}| \leq \kappa_m (\norm{h}_2 + \norm{f}_2)^2 \quad \text{for all  } (h, f) \in \C .
    $$
\end{itemize}

\noindent If the incoherence between these two column spaces is sufficiently small such that $4 \kappa_m < \kappa_r$, then we can conclude that $\norm{Xh+f}_2^2 \geq (\kappa_r - 2 \kappa_m)(\norm{h}_2 + \norm{f}_2)^2$. The small mutual incoherence property is especially important since it provides how the regression separates itself away from the sparse error.

3) To simplify our result, we consider a special case of the uniform Gaussian design, in which $\Sigma = I_{p \times p}$. In this situation, $C_{\min} = C_{\max} = \xi(\Sigma) = 1$. We have the following result which is a corollary of Theorem \ref{thm::parameter estimation} and Lemma \ref{lem::extended RE with Gaussian matrix}

\begin{cor} [\textrm{Standard Gaussian design}]
\label{cor::parameter estimation}
Let $X$ be a standard Gaussian design matrix. Consider the optimal solution $(\widehat{\beta}, \widehat{e})$ to the optimization problem (\ref{opt::lasso with sparse noise}) with regularization parameters chosen as
\begin{equation}
\lambda_{n,\beta} = \frac{4}{\gamma} \sqrt{\frac{\sigma^2 \log p}{n}} \quad \text{and} \quad \lambda_{n,e} = 4\sqrt{\frac{\sigma^2 \log n}{n}},
\end{equation}
for $\gamma \in (0,1]$. Also, assuming that $n \geq C k \log p$ and $s \leq \min \{ C_1 \frac{n}{\gamma^2 \log n}, C_2 n \}$ for some small constants $C, C_1, C_2$, Then with probability greater than $1 - c_1 \exp(-c_2 n)$, the error set $(h,f) = (\widehat{\beta} - \beta^{\star}, \widehat{e}-e^{\star})$ is bounded by
\begin{equation}
\label{inq::L2 error bound}
\norm{h}_2 + \norm{f}_2 \leq 384 \left( \frac{1}{\gamma}\sqrt{\frac{\sigma^2 k \log p}{n}} + \sqrt{\frac{\sigma^2 s \log n}{n}} \right).
\end{equation}
\end{cor}

Corollary \ref{cor::parameter estimation} reveals a remarkable result: by setting $\gamma = 1/\sqrt{\log n}$, even when the fraction of corruption is linearly proportional with the number of samples $n$, the extended Lasso (\ref{opt::lasso with sparse noise}) is still capable of recovering both coefficient vector $\beta^{\star}$ and corruption (missing) vector $e^{\star}$ within a bounded error (\ref{inq::L2 error bound}). Without the dense noise $w$ in the observation model (\ref{eqt::extend linear observations}) ($\sigma = 0$), the extended Lasso actually recovers the exact solution. This is impossible to achieve with the standard Lasso. Furthermore, if we know in prior that the number of corrupted observations is on the order of $\oper O(n/\log p)$, then selecting $\gamma = 1$ instead of $1/\log n$ will minimize the estimation error (see equation (\ref{inq::L2 error bound})) of Theorem \ref{thm::parameter estimation}.


\subsection{Feature selection with random Gaussian design}

In many applications, the feature selection criterion is more preferred \cite{Wainwright_Lasso_2009_J} \cite{ZY_Lasso_2006_J}. Feature selection refers to the property that the recovered parameter has the same signed support as the true regressor. In general, good feature selection implies good parameter estimation but the reverse direction does not usually hold. In this part, we investigate conditions for the design matrix and the scaling of $(n,p,k,s)$ such that both regression and sparse error vectors satisfy these criterion.

Consider the linear model (\ref{eqt::extend linear observations}) where $X$ is the Gaussian random design matrix whose rows are i.i.d zero mean with covariance matrix $\Sigma$. It has been well known in the Lasso that in order to obtain feature selection accuracy, the covariance matrix $\Sigma$ must obey two properties: invertibility and small mutual incoherence restricted on the set $T$. The first property guarantees that (\ref{opt::lasso with sparse noise}) is strictly convex, leading to the unique solution of the convex program, while the second property requires the separation between two components of $\Sigma$, one related to the set $T$ and the other to the set $T^c$ must be sufficiently small.
\begin{enumerate}
	\item \textbf{Invertibility.} To guarantee uniqueness, we require $\Sigma_{TT}$ to be invertible. Particularly, let $C_{\min} = \lambda_{\min} (\Sigma_{TT})$, we require $C_{\min} > 0$.
	
	\item \textbf{Mutual incoherence.} For some $\gamma \in (0,1)$,
		\begin{equation}
		\label{inq::incoherence assumption}
		\norm{\Sigma^*_{T^c T} (\Sigma_{TT})^{-1} }_{\infty} \leq (1 - \gamma).
		\end{equation}
    It is worth noting that these two invertibility and mutual incoherence properties are exactly the same as the conditions used to establish the exact signed support recovery in the standard Lasso (e.g \cite{Tropp_Relax_2006_J}, \cite{Wainwright_Lasso_2009_J}, \cite{ZY_Lasso_2006_J}).
\end{enumerate}

Toward the end, we will also elaborate on three other quantities operating on the restricted covariance matrix $\Sigma_{TT}$: $C_{\max}$, which is defined as the maximum eigenvalue of $\Sigma_{TT}$: $C_{\max} := \lambda_{\max} (\Sigma_{TT})$; and $D^-_{\max}$ and $D^+_{\max}$, which are denoted as $\ell_{\infty}$-norm of matrices $\Sigma_{TT}^{-1}$ and $\Sigma_{TT}$ : $D^-_{\max} := \norm{ (\Sigma_{TT})^{-1} }_{\infty}$ and $D^+_{\max} := \norm{ \Sigma_{TT} }_{\infty}$.

Our result also involves two other quantities operating on the conditional covariance matrix of $(X_{T^c} | X_T)$ defined as
\begin{equation}
\label{eqt::Sigma_Tc|T}
\Sigma_{T^c|T} := \Sigma_{T^c T^c} - \Sigma_{T^c T} \Sigma_{TT}^{-1} \Sigma_{T T^c}.
\end{equation}
They are defined as $\rho_u (\Sigma_{T^c|T}) = \max_i (\Sigma_{T^c|T})_{ii}$ and $\rho_l (\Sigma_{T^c|T}) = \frac{1}{2} \min_{i \neq j} [(\Sigma_{T^c|T})_{ii} + (\Sigma_{T^c|T})_{jj} - 2 (\Sigma_{T^c|T})_{ij}]$, which we often denote with the shorthand notation $\rho_u$ and $\rho_l$.

We establish the following result for Gaussian random design whose covariance matrix $\Sigma$ obeys the two assumptions.

\begin{thm} [\textrm{Achievability}]
\label{thm::Achievability - Gaussian design}  Given the linear model (\ref{eqt::extend linear observations}) with random Gaussian design and the covariance matrix $\Sigma$ satisfying invertibility and incoherence properties for any $\gamma \in (0,1)$, suppose that we solve the extended Lasso (\ref{opt::lasso with sparse noise}) with regularization parameters obeying
\begin{equation}
\label{inq::equation of lambda1}
\lambda_{n,\beta} = \frac{8}{\gamma} \sqrt{\frac{\sigma^2 \eta \log n \log p}{n} \max\{\rho_u, D^+_{\max} \} }
\end{equation}
\begin{equation}
\label{inq::equation of lambda2}
\text{and}\quad\quad \lambda_{n,e} = 4 \sqrt{\frac{\sigma^2 \log n}{n}},
\end{equation}


\noindent for some $\eta \in [\frac{1}{\log n}, 1)$. Assume that the sequence (n, p, k, s) and regularization parameters $\lambda_{n,\beta}$, $\lambda_{n,e}$ satisfying $s \leq \eta n $ and $n > \max \{ n_1, n_2 \}$ where $n_1$ and $n_2$ are defined as
\begin{equation*}
n_1 := \frac{4(1+\epsilon)}{(1-\eta)}\frac{\rho_u}{C_{\min} \gamma^2 } k \log(p-k)\left\{ \frac{9}{4} + (1-\eta)^2 \frac{\sigma^2 C_{\min} }{ \lambda_{n,\beta}^2 k }\right\}
\end{equation*}
\begin{multline*}
\text{and} \quad n_2 :=  48 (1+\epsilon) \frac{ \eta }{(1-\eta)^2} \frac{\max \{\rho_u, D^+_{\max} \}}{C_{\min} \gamma^2}  \\
\times \left( 1 - \frac{2\sigma \sqrt{\log n}}{\lambda_{n,e} \sqrt{n}} \right)^{-2} k \log(p-k) \log n
\end{multline*}
for $\epsilon \in (0,1)$. In addition, suppose that $\min_{i\in T} |\beta^{\star}_i| > f_{\beta} (\lambda_{n,\beta})$ and $\min_{i\in S} |e^{\star}_i| > f_{e} (\lambda_{n,\beta},\lambda_{n,e})$ where
\begin{equation}
\label{eqt::f-beta}
f_{\beta} := c_1  \lambda'_{n,\beta} + 20 \sqrt{\frac{\sigma^2 \log k}{C_{\min} (n-s)}} \quad\text{and  }
\end{equation}
\begin{equation}
\label{eqt::f-e}
f_{e} := c_2 \lambda'_{n,\beta} \sqrt{C_{\max}} \sqrt{\frac{sk + k\sqrt{sk}}{n}} + c_3\lambda_{n,e}
\end{equation}
\begin{equation*}
\text{with  }  \lambda'_{n,\beta} := \lambda_{n,\beta} \sqrt{\frac{k \log (p-k)}{(1-\eta)^2 n}} \norm{\Sigma^{-1/2}_{TT}}^2_{\infty}.
\end{equation*}

\noindent Then, the following properties holds with probability greater than $1 - c \exp(-c' \max\{\log n,\log (p-k)\})$:
\begin{enumerate}
	\item The solution pair ($\widehat{\beta}, \widehat{e}$) of the extended Lasso (\ref{opt::lasso with sparse noise}) is unique and has the exact signed support.
	\item $\ell_{\infty}$-norm bounds: $\norm{\widehat{\beta} - \beta^{\star}}_{\infty} \leq f_{\beta} (\lambda_{n,\beta})$ and $\norm{\widehat{e} - e^{\star}}_{\infty} \leq f_e (\lambda_{n,\beta}, \lambda_{n,e})$.
\end{enumerate}
\end{thm}

There are several interesting observations from the theorem.

1) The first important observation is that the extended Lasso is robust to arbitrarily large and sparse error observation. Under the same invertibility and mutual incoherence assumptions on the covariance matrix $\Sigma$ as the standard Lasso, the extended Lasso program can recover both the regression vector and error with exact signed supports even when almost all the observations are contaminated by arbitrarily large error with unknown support. What we sacrifice for the corruption robustness is an additional $\log$ factor to the number of samples. We notice that when the error fraction is $\oper O(n / \log n)$, only $\Omega (k \log (p-k))$ samples are sufficient to recover the exact signed supports of both the regression and sparse error vectors.

2) We consider the special case with Gaussian random design in which the covariance matrix $\Sigma = I_{p \times p}$. In this case, entries of $X$ is i.i.d. $\oper N(0,1)$ and we have quantities $C_{\min} = C_{\max} = D^+_{\max} = D^-_{\max} = \rho_u = \rho_l = 1$. In addition, the invertibility and mutual incoherence properties are automatically satisfied with $\gamma = 1$. The theorem implies that when the number of errors $s$ is arbitrarily close to $n$, the number of samples $n$ needed to recover the exact signed supports obeys $\frac{n}{\log n} = \Omega (k \log(p-k))$. Furthermore, Theorem \ref{thm::Achievability - Gaussian design} guarantees consistency in element-wise $\ell_{\infty}$-norm of the estimated regression at the rate of
$$
\norm{\widehat{\beta} - \beta^{\star} }_{\infty} = \oper O \left( \sqrt{\frac{\sigma^2 \log p}{n}} \sqrt{\frac{\eta k \log n \log (p-k)}{n}} \right).
$$

As $\eta$ is chosen to be $1/\sqrt{\log n}$ (equivalent to establish $s$ close to $n/\log n$), the $\ell_{\infty}$ error rate is on the order of $\oper O (\sigma \sqrt {\frac{\log p}{n}})$, which is known to be the same as that of the standard Lasso. On the other hand, if we select $\eta$ is arbitrarily close to unity $-$ equivalently, $s$ is close to $n$, the $\ell_{\infty}$ error rate is on the order of $\oper O (\sigma \sqrt {\frac{\log n \log p}{n}})$. This is naturally interpreted as the more fraction of corruption is on the observations, the higher reconstruction error we expect to get. What interesting is that we draw an explicit connection between the fraction of corruption and the reconstruction error obtained by the extended Lasso optimization.



3) Corollary \ref{cor::parameter estimation}, though interesting, is not able to guarantee stable recovery when the fraction of corruption converges to unity. We show in Theorem \ref{thm::Achievability - Gaussian design} that this fraction can come arbitrarily close to unity by sacrificing a factor of $\log n$ for the number of samples. Theorem \ref{thm::Achievability - Gaussian design} also implies that there is a significant difference between recovery to obtain small parameter estimation error versus recovery to obtain correct variable selection. When the amount of corrupted observations is linearly proportional to $n$, recovering the exact signed supports require an increase from $\Omega(k \log p)$ (in Corollary \ref{cor::parameter estimation}) to $\Omega (k \log p \log n)$ samples (in Theorem \ref{thm::Achievability - Gaussian design}). This behavior is captured similarly by the standard Lasso, as pointed out in the discussion after Corollary 2 of \cite{Wainwright_Lasso_2009_J}.


Our next theorem show that the number of samples needed to recover accurately the signed support is actually optimal. That is, whenever the rescaled sample size satisfies a certain threshold, regardless of what the regularization parameters $\lambda_{n,\beta}$ and $\lambda_{n,e}$ are selected, no solution of the extended Lasso can correctly identify the signed supports with high probability.

\begin{thm} [\textrm{Inachievability}]
\label{thm::Inachievability - Gaussian design} Given the linear model (\ref{eqt::extend linear observations}) with random Gaussian design and the covariance matrix $\Sigma$ satisfying invertibility and incoherence properties for any $\gamma \in (0,1)$. Let $\eta, \delta \in (0,1)$ and the sequence $(n,p,k,s)$ satisfies $s \geq \eta n$ and $n < \max\{n_1, n_2\}$ where $n_1$ and $n_2$ are defined as
\begin{equation*}
n_1 := \frac{2 (1-\delta)}{(1-\eta)} \frac{\rho_l k \log(p-k)}{C_{\max}(2-\gamma)^2}   \left\{\frac{3}{8} + (1-\eta)^2 \frac{\sigma^2 C_{\max}}{\lambda_{n,\beta}^2 k}  \right\}.
\end{equation*}
\begin{equation*}
\begin{split}
n_2 := &\frac{(1-\delta)}{12} \frac{\eta}{(1-\eta)^2} \frac{\rho_l}{C_{\max}} \\
&\times \left( 1 + \frac{2\sqrt{\sigma^2 \log n}}{\lambda_{n,e} \sqrt{n}} \right)^{-2} k \log(n-s) \log(p-k).
\end{split}
\end{equation*}
%
Then, with probability tending to unity, no solution pair of the extended Lasso (\ref{opt::extended L1}) has the correct signed support.
\end{thm}

When the covariance matrix of the design matrix $X$ is $\Sigma = I_{p \times p}$, or equivalently, entries of $X$ are i.i.d. Gaussian $\oper N(0,1)$. In addition, assume the regularization parameters $\lambda_{\beta},n$ and $\lambda_{n,e}$ are chosen from the families of (\ref{inq::equation of lambda1}) and (\ref{inq::equation of lambda2}), respectively. That is, $\lambda_{\beta,n} = 8\sqrt{\frac{\sigma^2 \eta \log n \log p}{n}}$ and $\lambda_{e,n} = 4 \sqrt{\frac{\sigma^2 \log n}{n}}$. Then, the theorem implies that the extended Lasso (\ref{opt::extended L1}) is not able to achieve the correct signed support solution whenever the number of observations is less than
\begin{align*}
n &\leq \max \left\{ c_1 \frac{1}{1-\eta} k \log(p-k), \right. \\
&c_2 \left. \frac{ \eta}{(1-\eta)^2} k \log(p-k) \log(1-\eta)n  \right\}.
\end{align*}

\section{Illustrative simulations}
\label{sec::simulations}

In this section, we provide several simulations to illustrate the capability of the extended Lasso in recovering the exact regression signed support when a significant fraction of observations is corrupted by large error. Simulations are performed for a range of parameters $(n, p, k,s)$ where the design matrix $X$ is uniform Gaussian random whose rows are i.i.d. $\oper N(0, I_{p \times p})$. For each fixed set of $(n,p,k,s)$, we generate sparse vectors $\beta^{\star}$ and $e^{\star}$ where locations of nonzero entries are distributed uniformly at random and their magnitudes are also Gaussian distributed.

In our experiments, we consider varying problem sizes $p = \{128, 256, 512 \}$ and three types of regression sparsity indices: sublinear sparsity ($k = 0.2 p /\log(0.2 p)$), linear sparsity ($k = 0.1 p$) and fractional power sparsity ($k = 0.5 p^{0.75}$). In all cases, we fixed the error support size $s = n/2$. This means half of the observations is corrupted. By this selection, Theorem \ref{thm::Achievability - Gaussian design} suggests that we require the number of samples to be $n \geq 2 C k \log (p-k) \log n$ to guarantee exact signed support recovery. We choose $\frac{n}{\log n} = 4 \theta k \log(p-k)$ where parameter $\theta$ is the rescaled sample size. This parameter control the success/failure of the extended Lasso.

In the algorithm, we select $\lambda_{n,\beta} = 2\sqrt{\frac{\sigma^2 \log p \log n}{n}}$ and $\lambda_{n,e} = 2\sqrt{\frac{\sigma^2 \log n}{n}}$ as suggested by Theorem \ref{thm::Achievability - Gaussian design}, where the noise level $\sigma =0.1$ is fixed. The algorithm reports a success if the solution pair has the same signed support as $(\beta^{\star}, e^{\star})$. In Fig. \ref{fig::probability}, each point on the curve represents the average of $100$ trials.

As demonstrated by the simulation results, our extended Lasso is capable of recovering the exact signed support of both $\beta^{\star}$ and $e^{\star}$ even $50\%$ of the observations are contaminated. Furthermore, up to unknown constants, our Theorem \ref{thm::Achievability - Gaussian design} and \ref{thm::Inachievability - Gaussian design} match with simulation results. As the sample size $\frac{n}{\log n} \leq 2 k \log(p-k)$, the probability of success starts diving down to zero, implying the failure of the extended Lasso.

\begin{figure}[h]
\centering
\subfigure{
\includegraphics[width=2.7in]{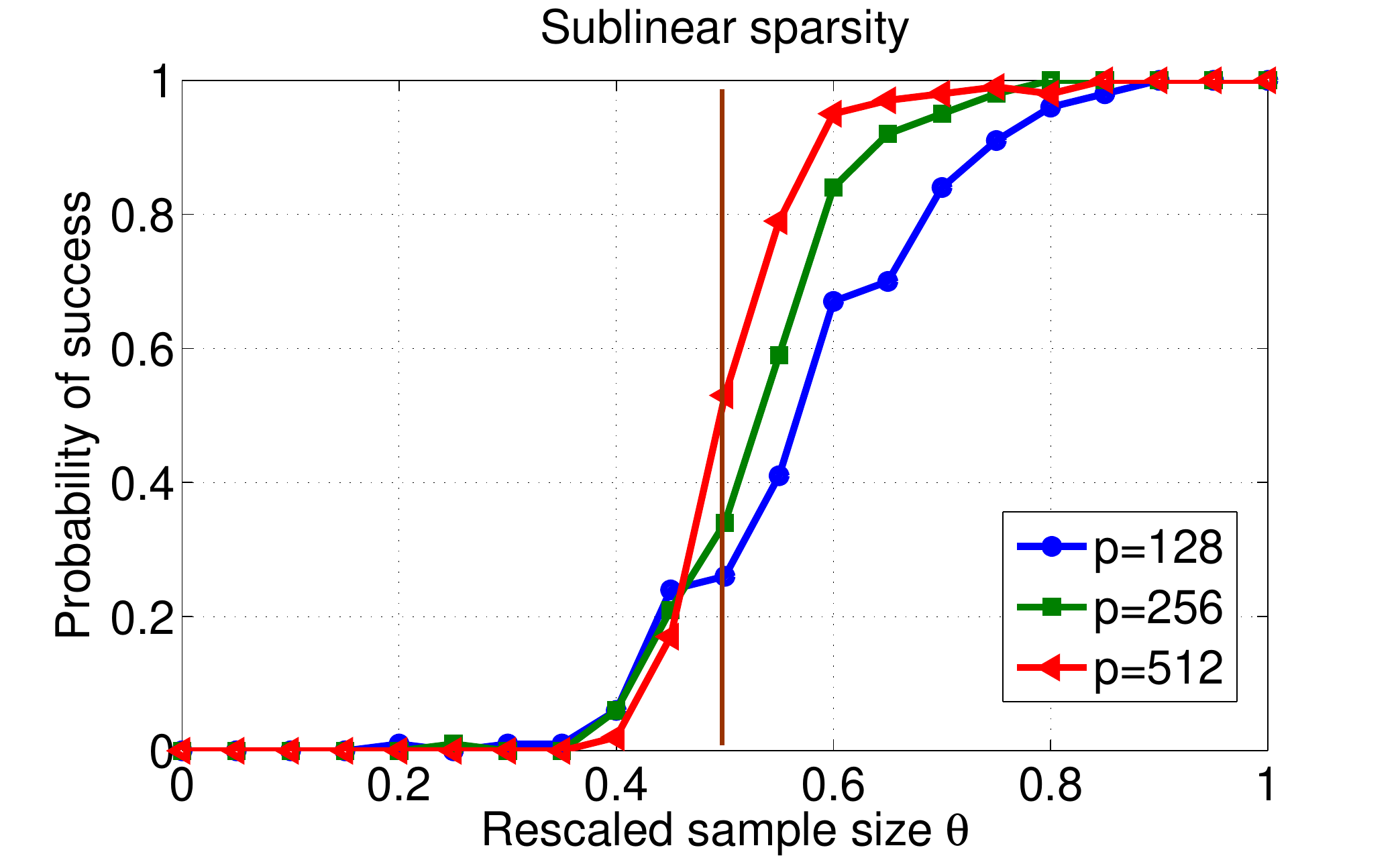} }
\subfigure{
\includegraphics[width=2.7in]{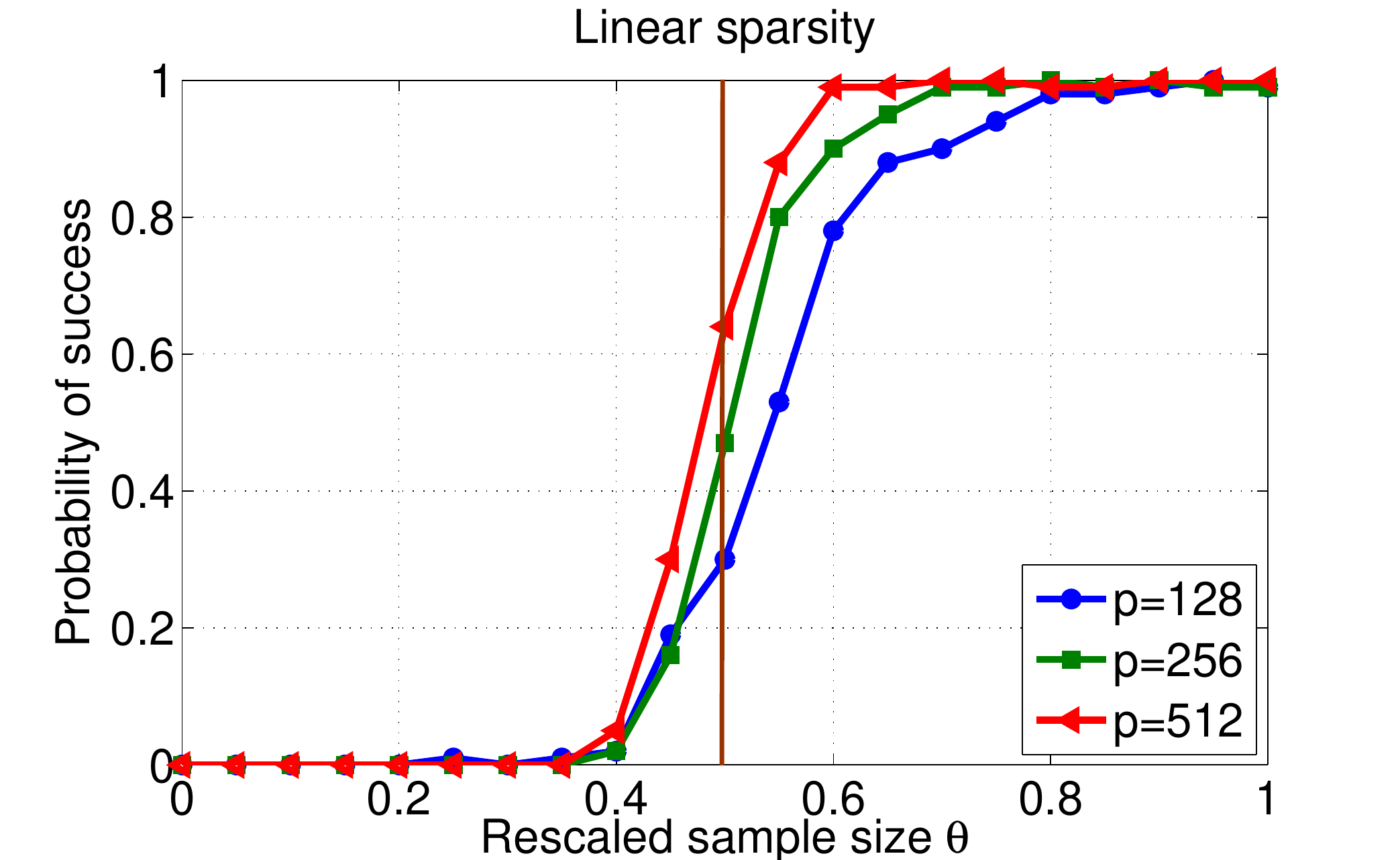} }
\subfigure{
\includegraphics[width=2.7in]{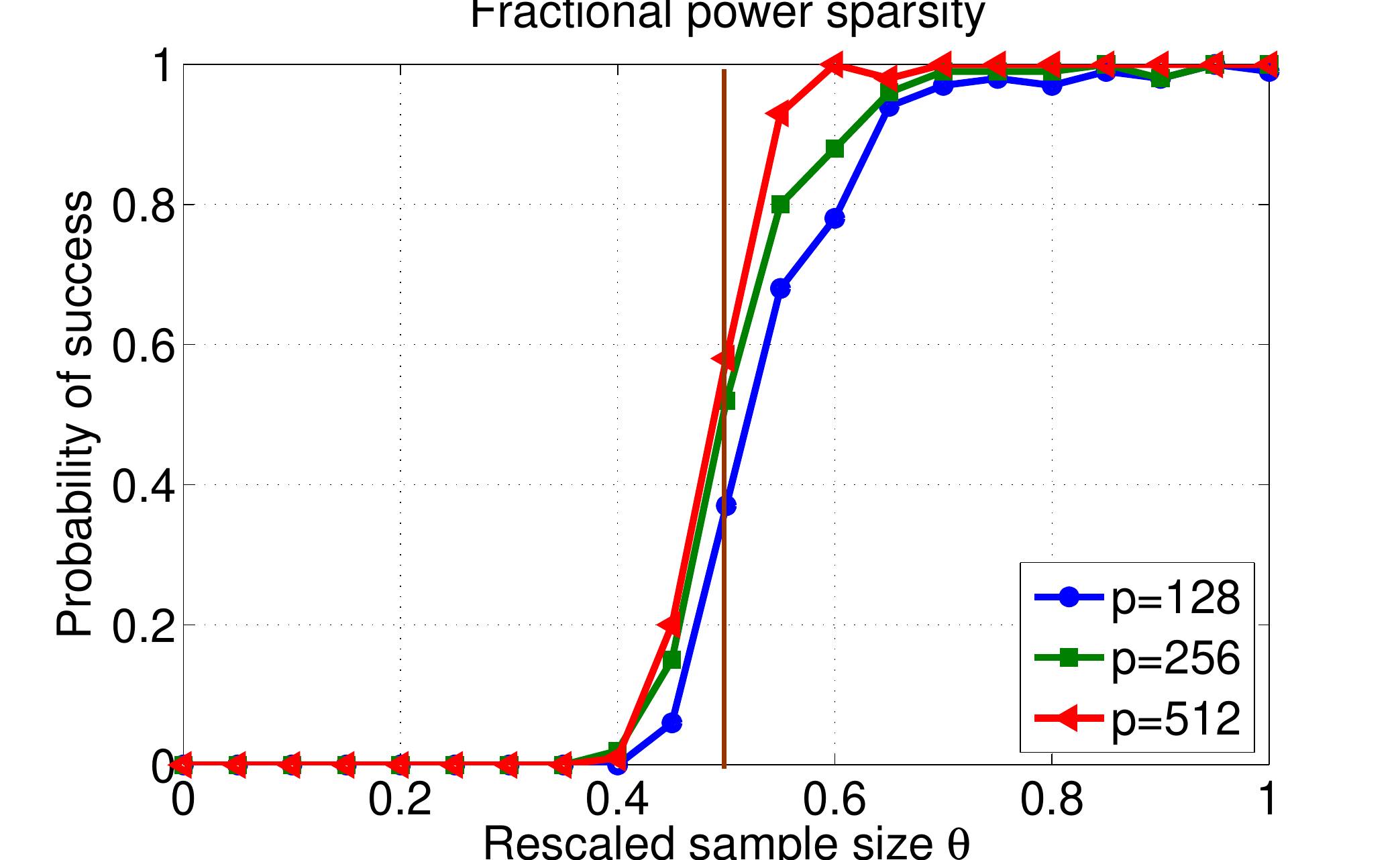} }
\caption{\small Probability of success in recovering the signed supports}
\label{fig::probability}
\end{figure}

\section{Proof of Theorem \ref{thm::parameter estimation} and related results}
\label{sec::proof 1}

\begin{proof} [\textit{Proof of Theorem \ref{thm::parameter estimation}}]
Since $(\widehat{\beta}, \widehat{e})$ is the pair of the optimal solution of (\ref{opt::lasso with sparse noise}), we have
\begin{equation}
\begin{split}
\label{inq::inequality of optimal solution - 1st}
&\frac{1}{2n} \norm{y - X \widehat{\beta} - \sqrt{n}\widehat{e} }_2^2 + \lambda_{n,\beta} \norm{\widehat{\beta}}_1 + \lambda_{n,e} \norm{\widehat{e}}_1 \\
&\leq \frac{1}{2n} \norm{y - X \beta^{\star} - \sqrt{n} e^{\star}}_2^2 + \lambda_{n,\beta} \norm{\beta^{\star}}_1 + \lambda_{n,e} \norm{e^{\star}}_1.
\end{split}
\end{equation}

\noindent From $h = \widehat{\beta} - \beta^{\star}$ and $f = \widehat{e} - e^{\star}$, we can easily see that
\begin{align*}
\norm{y - X \widehat{\beta} - \sqrt{n}\widehat{e} }_2^2 &= \norm{y - X \beta^{\star} - \sqrt{n} e^{\star}}_2^2 \\
&- 2\inner{w, X h + \sqrt{n}f} + \norm{X h + \sqrt{n}f}_2^2.
\end{align*}

\noindent Moreover, it is clear that
\begin{equation*}
\begin{split}
\norm{\beta^{\star}}_1 - \norm{\widehat{\beta}}_1  &= \norm{\beta^{\star}}_1 - \norm{\beta^{\star} + h}_1 \\
&= \norm{\beta^{\star}}_1 - \norm{\beta^{\star} + h_T}_1 - \norm{h_{T^c}}_1 \\
&\leq \norm{h_T}_1 - \norm{h_{T^c}}_1.
\end{split}
\end{equation*}

\noindent We also have a similar bound with $e$
\begin{equation*}
\norm{e^{\star}}_1 - \norm{\widehat{e}}_1 \leq \norm{f_S}_1 - \norm{f_{S^c}}_1.
\end{equation*}

\noindent Putting these pieces into (\ref{inq::inequality of optimal solution - 1st}) we can bound
\begin{equation}
\label{inq::inequality of optimal solution - 2nd}
\begin{split}
&\frac{1}{2n} \norm{X h + \sqrt{n}f}_2^2 \\
&\leq \frac{1}{n}\inner{w, X h + \sqrt{n}f} + \lambda_{n,\beta} (\norm{h_T}_1 - \norm{h_{T^c}}_1) \\
&+ \lambda_{n,e} ( \norm{f_S}_1 - \norm{f_{S^c}}_1 ) \\
&\leq \frac{1}{n}\norm{X^* w}_{\infty} \norm{h}_1 + \frac{1}{\sqrt{n}} \norm{w}_{\infty} \norm{f}_1 \\
&+ \lambda_{n,\beta} (\norm{h_T}_1 - \norm{h_{T^c}}_1) + \lambda_{n,e} ( \norm{f_S}_1 - \norm{f_{S^c}}_1 ) \\
&\leq \left( \frac{1}{n} \norm{X^* w}_{\infty} + \lambda_{n,\beta} \right) \norm{h_T}_1 - \left( \lambda_{n,\beta} - \frac{1}{n}\norm{X^* w}_{\infty} \right) \norm{h_{T^c}}_1 \\
&{ }+ \left( \frac{1}{\sqrt{n}} \norm{w}_{\infty} + \lambda_{n,e} \right) \norm{f_S}_1 - \left(\lambda_{n,e} - \frac{1}{\sqrt{n}} \norm{w}_{\infty} \right) \norm{f_{S^c}}_1.
\end{split}
\end{equation}

\noindent By the choices of $\lambda_{n,\beta}$ and $\lambda_{n,e}$ in the lemma, we have $\frac{1}{n}\norm{X^* w}_{\infty} \leq \frac{\gamma}{2} \lambda_{n,\beta} \leq \frac{\lambda_{n,\beta}}{2}$ and $\frac{1}{\sqrt{n}} \norm{w}_{\infty} \leq \frac{\lambda_{n,e}}{2}$. Therefore,
\begin{align*}
\frac{1}{2n} \norm{X h + \sqrt{n} f}_2^2 &\leq \lambda_{n,\beta} \frac{3}{2} \norm{h_T}_1 - \frac{\lambda_{n,\beta}}{2} \norm{h_{T^c}}_1 \\
&+ \frac{3}{2} \lambda_{n,e}  \norm{f_S}_1 - \frac{1}{2} \lambda_{n,e}  \norm{f_{S^c}}_1 .
\end{align*}

\noindent The left-hand side is greater than zero, thus the error pair $(h,f)$ belongs to the set $\C$ defined in (\ref{eqt::set C}). Hence, by the extended RE,
\begin{align*}
\kappa^2_l (\norm{h}_2 + \norm{f}_2)^2 &\leq 3 \lambda_{n,\beta} \norm{h_T}_1 + 3 \lambda_{n,e} \norm{f_S}_1 \\
&\leq 3 \lambda_{n,\beta} \sqrt{k}\norm{h}_2 + \lambda_{n,e} \sqrt{s} \norm{f}_2 ,
\end{align*}
where the last inequality follows from the crude $\ell_1/\ell_2$ bound: $\norm{h_T}_1 \leq \sqrt{k} \norm{h}_2$. If $\lambda \sqrt{s/k} \leq 1$, the right-hand side is upper bounded by $3\lambda_{n,\beta} \sqrt{k} (\norm{h}_2 + \norm{f}_2)$. On the other hand, it is upper bounded by $3\lambda_{n,e} \sqrt{s} (\norm{h}_2 + \norm{f}_2)$ if $\lambda \sqrt{s/k} \geq 1$. Combining these pieces together, we conclude
$$
\norm{h}_2 + \norm{f}_2 \leq 3\kappa_l^{-2} \max \left\{ \lambda_{n,\beta} \sqrt{k}, \lambda_{n,e} \sqrt{s} \right\},
$$
which completes our proof.
\end{proof}

\begin{proof} [\textit{Proof of Lemma \ref{lem::extended RE with Gaussian matrix}}]
Decompose $\frac{1}{n}\norm{Xh + \sqrt{n} f}_2^2 = \frac{1}{n}\norm{Xh}_2^2 + \norm{f}_2^2 + \frac{2}{\sqrt{n}} \inner{Xh, f}$. In order to lower bound the left-hand side, our main tool is to control the lower bound of each term on the right-hand side.

To establish a lower bound of $\frac{1}{n}\norm{Xh}_2^2$, we leverage an appealing result of \cite{RWY_RE_2010_J}. This result stated that for any Gaussian random matrix $X$ with i.i.d. $\oper N(0, \Sigma)$ rows, there exists universal positive constants $c_1, c_2$ such that the following inequality holds with probability greater than $1 - c_1 \exp(-c_2 n)$
\begin{equation}
\label{inq::lower bound l2 of Xh}
\frac{1}{\sqrt{n}}\norm{X v}_2 \geq \frac{\sqrt{C_{\min}}}{4} \norm{v}_2 - 9 \sqrt{\xi (\Sigma)} \sqrt{\frac{\log p}{n}} \norm{v}_1 \end{equation}
for $\forall v \in \R^p$. Here, we remind the reader of the notation $\xi (\Sigma) = \max_{j=1,...,d} \Sigma_{jj}$ and $C_{\min} = \lambda_{\min} (\Sigma)$.

\noindent We now apply this inequality for the error vector $h$ in the set $\C$. Since $h \in \C$, we have
\begin{align*}
\norm{h}_1 &\leq 4 \norm{h_T}_1 + 3 \lambda \norm{f_S}_1 \leq 4 \sqrt{k} \norm{h}_2 + 3 \lambda \sqrt{s} \norm{f}_2.
\end{align*}

\noindent Next taking advantage of (\ref{inq::lower bound l2 of Xh}) yields
\begin{align*}
\frac{1}{\sqrt{n}}\norm{Xh}_2 &\geq \left( \frac{\sqrt{C_{\min}}}{4} - 36 \sqrt{\frac{\xi k \log p}{n}} \right) \norm{h}_2 \\
&- 27 \lambda \sqrt{\frac{\xi s \log p}{n}} \norm{f}_2.
\end{align*}
where we denote the shorthand notation $\xi := \xi(\Sigma)$. This inequality leads to
\begin{align*}
\frac{1}{\sqrt{n}}\norm{X h}_2 + \norm{f}_2 &\geq \left( \frac{\sqrt{C_{\min}}}{4} - 36 \sqrt{\frac{\xi k \log p}{n}} \right) \norm{h}_2 \\
&+ \left( 1 - 27 \lambda \sqrt{\frac{\xi s \log p}{n}} \right) \norm{f}_2.
\end{align*}

\noindent From the assumptions of the lemma and the choice of $\lambda$ in (\ref{eqt::lambda_n}), the two quantities in the brackets are strictly greater than $0$. Thus, $\frac{1}{\sqrt{n}} \norm{X h}_2 + \norm{f}_2 \geq \frac{1}{2} (\norm{h}_2 + \norm{f}_2)$; or equivalent $\frac{1}{n}\norm{X h}^2_2 + \norm{f}^2_2 \geq \frac{1}{8} (\norm{h}_2^2 + \norm{f}_2^2)$.

\noindent In combination with the following lemma \ref{lem::bound inner product of Xh and f}, we conclude that
$$
\frac{1}{n}\norm{Xh + \sqrt{n} f}^2_2 \geq \frac{1}{16} (\norm{h}_2^2 + \norm{f}_2^2),
$$
as claimed.
\end{proof}

\begin{lem}
\label{lem::bound inner product of Xh and f}
Consider the random Gaussian design matrix $X$ whose rows are i.i.d. $\oper N(0,\Sigma)$. Assume that $n^2 C_{\max} \xi(\Sigma) = \Theta(1)$. Suppose that $s \leq C_1 \frac{k \log p}{\gamma \log n}$ and $n \geq C_2 k \log p$, then the following inequality holds with probability greater than $1 - \exp(-c n)$
\begin{equation*}
\frac{2}{\sqrt{n}} |\inner{Xh, f}| \leq \frac{1}{16} (\norm{h}_2^2 + \norm{f}_2^2).
\end{equation*}
\end{lem}
\begin{proof}
Divide the set $T^c$ into subset $T_1,T_2,..., T_q$ of size $k$ such that the first set $T_1$ contains $k$ entries of $h$ indexed by $T$, the set $T_2$ contains $k$ largest absolute entries of the vector $h_{T^c}$, $T_3$ contains the second $k$ largest absolution entries of $h_{T^c}$ and so on. By the same strategy, we also divide the set $S^c$ into subset $S_1, S_2,..., S_r$ such that the first set $S_1$ contains $s$ entries of $f$ indexed by $S$ and sets $S_2, S_3,...$ are of size $s' \geq s$.

\noindent We now have
\begin{align*}
\frac{1}{\sqrt{n}} |\inner{Xh,f}| &\leq \sum_{i,j} \frac{1}{\sqrt{n}} |\inner{X_{S_i T_j} h_{T_j}, f_{S_i} } | \\
&\leq \max_{ij} \frac{1}{\sqrt{n}} \norm{X_{S_i T_j}} \sum_{ij} \norm{h_{T_j}}_2 \norm{f_{S_i}}_2.
\end{align*}

\noindent Notice that the matrix $X_{S_i T_j}$ is the random Gaussian matrix whose rows are $\oper N(0, \Sigma_{T_j T_j})$. By the random Gaussian matrix concentration in Lemma \ref{lem::standard random matrix inequality of singular values} in Appendix \ref{appen::concentration inequalities}, we have with probability greater than $1- 2\exp(-\tau^2 n/2)$,
\begin{align*}
\frac{1}{\sqrt{n}} \norm{X_{S_i T_j}} \leq \norm{\Sigma^{1/2}_{T_j T_j}} \left( \sqrt{\frac{k}{n}}+\sqrt{\frac{s'}{n}} + \tau \right) .
\end{align*}

\noindent Taking the union bound over all possibility of $T_j$ and $S_i$, we have this inequality holds with probability at least $1- 2 \binom{n}{s} \binom{p}{k}\exp(-\tau^2 n/2) $. Assuming that $n \geq c^{-1}_1 k \log (p/k)$, we have $\binom{p}{k} \leq \left( \frac{e p}{k} \right)^k \leq \exp(c_1 n)$. In addition, assuming $n \geq c^{-1}_2 s' \log (n/s')$, we have $\binom{n}{s'} \leq \left(\frac{e n}{s'} \right)^s \leq \exp(c_2 n)$. Therefore, with sufficiently small constants $c_1$ and $c_2$, we get
\begin{equation*}
\max_{ij} \frac{1}{\sqrt{n}} \norm{X_{S_i T_j}} \leq \sqrt{C_{\max}}  \left(\sqrt{\frac{k}{n}}+\sqrt{\frac{s'}{n}} + \tau \right).
\end{equation*}
with probability at least $1 - \exp(- (\tau^2/2 -c_1 -c_2) n)$ where we recall the definition of $C_{\max} := \lambda_{\max} (\Sigma)$.

A standard bound in \cite{CRT_Stability_2006a_J} gives us: $\sum_{i=3}^q \norm{h_{T_i}}_2 \leq k^{-1/2} \norm{h_{T^c}}_1$. In addition, since $h$ belongs to the set $\C$, $\norm{h_{T^c}}_1 \leq 3 \sqrt{k} \norm{h}_2 + 3\lambda \sqrt{s} \norm{f}_2 $. Hence,
\begin{align*}
\sum_{i=1}^q \norm{h_{T_i}}_2 &\leq 2 \norm{h}_2 + \sum_{i=3}^q \norm{h_{T_i}}_2 \leq 5 \norm{h}_2 + 3 \lambda \sqrt{\frac{s}{k}} \norm{f}_2.
\end{align*}

\noindent Similar manipulations along with the choice of $s' \geq s$ also yields
$$
\sum_{i=3}^r \norm{f_{S_i}}_2 \leq s'^{-1/2} \norm{f_{S^c}}_1 \leq \frac{3}{\lambda} \sqrt{\frac{k}{s'}} \norm{h}_2 + 3\norm{f}_2,
$$
leading to
$$
\sum_{i=1}^r \norm{f_{S_i}}_2 \leq \frac{3}{\lambda} \sqrt{\frac{k}{s'}} \norm{h}_2 + 5 \norm{f}_2.
$$

\noindent Hence, $\frac{1}{\sqrt{n}}|\inner{Xh,f}| $ is upper bounded by
\begin{align*}
&C^{1/2}_{\max}  \left(\sqrt{\frac{k}{n}}+\sqrt{\frac{s'}{n}} + \tau \right) \\
&\times \left(5 \norm{h}_2 + 3 \lambda \sqrt{\frac{s'}{k}} \norm{f}_2 \right) \left(\frac{3}{\lambda} \sqrt{\frac{k}{s'}} \norm{h}_2 + 5 \norm{f}_2 \right) \\
&\leq 25 C^{1/2}_{\max}  \left(\sqrt{\frac{k}{n}}+\sqrt{\frac{s'}{n}} + \tau \right) \\
& \times \max \left\{ \lambda \sqrt{\frac{s'}{k}}, \frac{1}{\lambda} \sqrt{\frac{k}{s'}} \right\} (\norm{h}_2 + \norm{f}_2)^2.
\end{align*}

\noindent We select $s' = C \frac{k \log p}{\gamma^2 \log n}$ with an appropriate constant $C$. From the assumption that $C_{\max} \xi(\Sigma) = \Theta (1)$ and a few algebraic manipulations, we can show that $25 C^{1/2}_{\max} \max \left\{ \lambda \sqrt{\frac{s'}{k}}, \frac{1}{\lambda} \sqrt{\frac{k}{s'}} \right\} \leq c \frac{1}{\sqrt{n}}$. Therefore,
\begin{align*}
\frac{1}{\sqrt{n}} |\inner{Xh, f}| &\leq c \left(\sqrt{\frac{k}{n}} + \sqrt{\frac{s'}{n}} + \tau \right) (\norm{h}_2 + \norm{f}_2)^2 \\
&\leq \frac{1}{16} (\norm{h}_2 + \norm{f}_2)^2
\end{align*}
for sufficiently small $\tau$ and $n \geq C s'$.
\end{proof}

\section{Proof of Theorem \ref{thm::Achievability - Gaussian design} - Achievability}
\label{sec::proof 2}

By KKT condition, $\widehat{\beta}$ and $\widehat{e}$ is a pair of solution of (\ref{opt::lasso with sparse noise}) if and only if the following set of equations satisfies
\begin{eqnarray}
\label{eqt::KKT condition on beta}
-  \frac{1}{n} X^* (y - X \widehat{\beta} - \widehat{e}) + \lambda_{n,\beta} z^{(\beta)} &=& 0\\
\label{eqt::KKT condition on e}
- \frac{1}{\sqrt{n}}(y - X \widehat{\beta} - \widehat{e}) + \lambda_{n,e} z^{(e)} &=& 0,
\end{eqnarray}
where $z^{(\beta)}$ and $z^{(e)}$ are elements of the subgradients of the $\ell_1$ norm evaluated at $\widehat{\beta}$ and $\widehat{e}$, respectively. It has been well established that $(\widehat{\beta},\widehat{e})$ is the unique solution to the extended Lasso program if
\begin{equation}
\label{eqt::conditions on beta}
\begin{cases}
   \frac{1}{n} X^*_i (y - X \widehat{\beta} - \widehat{e}) = \lambda_{n,\beta} \sgn(\widehat{\beta}_i)  \quad & \text{for } \widehat{\beta}_i \neq 0 \\
    |z^{(\beta)}_i| = \frac{1}{n\lambda_{n,\beta}} | X^*_i (y - X \widehat{\beta} - \widehat{e}) | < 1 \quad  & \text{for } \widehat{\beta}_i = 0 .\\
\end{cases}
\end{equation}
and
\begin{equation}
\label{eqt::conditions on e}
\begin{cases}
    \frac{1}{\sqrt{n}}( y_i - X_i \widehat{\beta} - \widehat{e}_i) = \lambda_{n,e} \sgn(\widehat{e}_i)  \quad  & \text{for } \widehat{e}_i \neq 0 \\
    |z^{(e)}_i| = \frac{1}{\sqrt{n}\lambda_{n,e}} | y_i - X_i \widehat{\beta} - \widehat{e}_i | < 1 \quad  & \text{for } \widehat{e}_i = 0 ,\\
\end{cases}
\end{equation}

We will show that under the assumptions of Theorem \ref{thm::Achievability - Gaussian design}, the solution pair of the extended Lasso is given by $(\widehat{\beta}, \widehat{e}) = (\beta^{\star} + h,e^{\star} + g)$ where $h_{T^c} = 0$, $g_{S^c} = 0$ and
\begin{equation}
\begin{split}
\label{eqt::h_T}
h_T &= ( X^*_{S^c T} X_{S^cT} )^{-1} \\
&\times [ X^*_{S^c T} w_{S^c} + \sqrt{n} \lambda_{n,e} X^*_{S T} \sgn(e^{\star}_S) - n \lambda_{n,\beta} \sgn(\beta^{\star}_T) ],
\end{split}
\end{equation}
and
\begin{equation}
\begin{split}
\label{eqt::g_S}
g_S &= -\frac{1}{\sqrt{n}} X_{ST} ( X^*_{S^c T} X_{S^cT} )^{-1} \\
&\times [ X^*_{S^c T} w_{S^c} + \sqrt{n} \lambda_{n,e} X^*_{S T} \sgn(e^{\star}_S) - n \lambda_{n,\beta} \sgn(\beta^{\star}_T) ] \\
&+ \frac{1}{\sqrt{n}} w_S - \lambda_{n,e} \sgn(e^{\star}_S) .
\end{split}
\end{equation}

The expressions of $h_T$ and $g_S$ in the above equations are obtained by solving the KKT conditions (\ref{eqt::KKT condition on beta}) and (\ref{eqt::KKT condition on e}) restricted on $\widehat{\beta}_{T^c} = 0$ and $\widehat{e}_{S^c} = 0$ together with setting $z^{(\beta)}_T = \sgn(\beta^{\star}_T)$ and $z^{(e)}_S = \sgn(e^{\star}_S)$. We note that due to the conditions of the sample size $n$ and the fraction of errors in Theorem \ref{thm::Achievability - Gaussian design}, $X^*_{S^c T} X_{S^cT}$ is invertible thanks to the random Gaussian matrix concentration inequalities (see Lemma \ref{lem::standard random matrix inequality of singular values} in Appendix \ref{appen::concentration inequalities}). Therefore, the expressions of $h_T$ and $g_S$ are valid.

To confirm that $(\widehat{\beta}, \widehat{e})$ is the optimal solution of the extended Lasso (\ref{opt::lasso with sparse noise}), in the following subsections, we will check that $\widehat{\beta}$ and $\widehat{e}$ chosen above obey conditions (\ref{eqt::conditions on beta}) and (\ref{eqt::conditions on e}). In particular,
\begin{enumerate}
    \item In Subsection \ref{sec::upper Linf z^beta_Tc}, we show that $\norm{ z^{(\beta)}_{T^c} }_{\infty} < 1$.
    \item In Subsection \ref{sec::bound infty norm of z^e_S^c}, we show that $ \norm{ z^{(e)}_{S^c} }_{\infty} < 1$.
    \item In Subsection \ref{sec::upper Linf h_T}, we establish that $\norm{h_T}_{\infty} \leq f_{\beta}(\lambda_{n,\beta})$. It then follows from the assumptions of Theorem \ref{thm::Achievability - Gaussian design} that $\norm{h_T}_{\infty} < \min_{i\in T} |\beta^{\star}_i|$ and, therefore, $\text{supp}(\widehat{\beta}_T) = \text{supp}(\beta^{\star}_T)$ and $\sgn(\widehat{\beta}_T) = \sgn(\beta^{\star}_T)$.
    \item In Subsection \ref{sec::upper Linf f_S}, we establish that $\norm{g_S}_{\infty} \leq f_e(\lambda_{n,\beta}, \lambda_{n,e})$. It then follows from the assumptions Theorem \ref{thm::Achievability - Gaussian design} that $\norm{g_S}_{\infty} < \min_{i \in S} |e^{\star}_i|$ and, therefore, $\text{supp}(\widehat{e}_S) = \text{supp}(e^{\star}_S)$ and $\sgn(\widehat{e}_S) = \sgn(e^{\star}_S)$.
\end{enumerate}

\subsection{Verify the upper bound of $\norm{z^{(\beta)}_{T^c}}_{\infty}$}
\label{sec::upper Linf z^beta_Tc}

\begin{proof}
First, we define a notation which will be used throughout the rest of the paper. Let $\lambda : = \frac{\lambda_{n,e}}{\lambda_{n,\beta}}$. By the definition of $\lambda_{n,\beta}$ and $\lambda_{n,e}$ in (\ref{inq::equation of lambda1}), we have
\begin{equation}
\label{eqt::equation of lambda}
\lambda = \frac{\gamma}{ 2\sqrt{\max \{ \rho_u, D^+_{\max}  \}}}  \sqrt{\frac{1}{\eta \log p}},
\end{equation}
where we introduce another shorthand notation $\rho_u = \rho_u (\Sigma_{T^c|T})$.

From the expression of $\widehat{\beta} = \beta^{\star} + h$ and $\widehat{e} = e^{\star} + g$ with $h_{T^c}=0$, $g_{S^c} = 0$ and $h_T$, $g_S$ defined in (\ref{eqt::h_T}) and (\ref{eqt::g_S}), we substitute into $z^{(\beta)}_{T^c} = \frac{1}{\lambda_{n,\beta}} X^*_{T^c}(y - X \widehat{\beta} - \widehat{e})$ together with noticing that $X^*_{T^c} X_{T} - X^*_{S T^c} X_{ST} = X^*_{S^c T^c} X_{S^c T}$, $ X^*_{T^c} w  - X^*_{S T^c} w_S = X^*_{S^c T^c} w_{S^c}$ to arrive at
\begin{equation}
\label{eqt::final equation of z^beta_T^c}
\begin{split}
z^{(\beta)}_{T^c} &= \frac{1}{n \lambda_{n,\beta}} X^*_{S^c T^c} \Pi_{S^c T} w_{S^c} \\
&{ }-  X^*_{S^c T^c} X_{S^c T} ( X^*_{S^c T} X_{S^cT} )^{-1} z \\
&{ }+ \frac{1}{\sqrt{n}}\lambda X^*_{S T^c} \sgn(e^{\star}_S).
\end{split}
\end{equation}
 Here, we define $\Pi_{S^c T} := I - X_{S^c T} ( X^*_{S^c T} X_{S^cT} )^{-1} X^*_{S^c T}$ which is an orthogonal projection onto the column space of $X_{S^c T}$ and $z := \frac{1}{\sqrt{n}}\lambda X^*_{S T} \sgn(e^{\star}_S) - \sgn(\beta^{\star}_T)$.

We can further simplify the expression of $z^{(\beta)}_{T^c}$ by denoting
\begin{equation}
\label{eqt::v}
v:= \left( \begin{array}{c}
 \frac{1}{\sqrt{n}} \lambda \sgn(e^{\star}_S) \\
 \frac{1}{ n \lambda_{n,\beta}} \Pi_{S^c T} w_{S^c} - X_{S^c T} ( X^*_{S^c T} X_{S^cT} )^{-1} z \\                                                                   \end{array}
                                                                              \right) ,
\end{equation}
then we have
\begin{equation}
\label{eqt::reformulate z^beta_Tc}
z^{(\beta)}_{T^c} = [X^*_{ST^c} \quad X^*_{S^c T^c}] v = X^*_{T^c} v.
\end{equation}

\noindent Conditioning on $X_{T}$, the matrix $X^*_{T^c}$ can be decomposed into a linear prediction plus a prediction error
\begin{equation}
X^*_{T^c} = \Sigma_{T^c T} \Sigma^{-1}_{TT} X^*_{T} + E^*_{T^c},
\end{equation}
where each row of the matrix $E_{S^c T^c}$ is a $\oper N(0, \Sigma_{T^c | T})$ Gaussian random vector whose entries are i.i.d and $\Sigma_{T^c | T}$ is defined in (\ref{eqt::Sigma_Tc|T}). Therefore, $z^{(\beta)}_{T^c}$ consists of two components in which the first is
\begin{equation*}
a := \Sigma_{T^c T} \Sigma^{-1}_{TT} X^*_{T}  v ,
\end{equation*}
and the second is
\begin{equation}
\label{eqt::define b}
b := E^*_{T^c} v.
\end{equation}

\noindent Since $\Pi_{S^c T}$ is the orthogonal projection onto the space spanned by columns of the matrix $X_{S^c T}$, we have $X^*_{S^c T} \Pi_{S^cT} = 0$. Thus, $a$ can be simplified as
\begin{equation}
\label{eqt::define a}
\begin{split}
a &=  \frac{1}{\sqrt{n}}\Sigma_{T^c T} \Sigma^{-1}_{TT} (\lambda X^*_{ST} \sgn(e^{\star}_S)) - \Sigma_{T^c T} \Sigma^{-1}_{TT} z \\
&= \Sigma_{T^c T} \Sigma^{-1}_{TT} \sgn(\beta^{\star}_T).
\end{split}
\end{equation}

\noindent The mutual incoherent assumption in (\ref{inq::incoherence assumption}) gives us $\norm{a}_{\infty} \leq 1 -\gamma$. All that left is to establish the $\ell_{\infty}$-norm of the second component: $\norm{b}_{\infty} \leq \gamma$. Denote $E_i$ as the $i$-th column of the matrix $E_{T^c}$ and condition on $X_{S^cT}$, the $i$-th coefficient of the vector $b$: $b_i = \inner{E_i, v}$ is a Gaussian random variable with variance $\Var(b_i) := v^* \E E_i E_i^* v \leq \rho_u \norm{v}_2^2$ where $\norm{v}_2^2$ is quantified as,
\begin{equation}
\label{eqt::M}
M :=  \frac{\lambda^2 s}{n} + \frac{1}{n^2\lambda_{n,\beta}^2} \norm{\Pi_{S^c T} w_{S^c}}_2^2 + z^* (X_{S^c T}^* X_{S^c T})^{-1} z.
\end{equation}

We state two supporting lemmas whose proof are deferred to the end of this section.
\begin{lem}
\label{lem::bound Linf norm of z}
Denote $z = \frac{1}{\sqrt{n}}\lambda X^*_{ST} \sgn(e_S^{\star}) - \sgn(\beta_T^{\star}) $. Define the event
$$
\oper E_z := \left\{ \norm{z}_{\infty} \leq \lambda \sqrt{ \frac{D^+_{\max} s \log p}{n}}  + 1 \right\}.
$$
Then, $\Prob (\oper E_z ) \geq 1 - 2\exp(-\log p)$.
\end{lem}

\begin{lem}
\label{lem::upper bound M}
For any $\epsilon \in (0,1)$, define the event $\overline{\oper E} = \{ M \leq \overline{M} \}$, where
\begin{equation}
\begin{split}
\label{eqt::M overline}
\overline{M} &:= \frac{1}{n}\lambda^2 s + \left( 1 + \max \left\{ \epsilon, 4 \sqrt{\frac{k}{n-s}} \right\} \right) \\
&\times \left( \frac{\sigma^2 (n-s)}{n^2 \lambda_{n,\beta}^2} + \frac{k \left(1 + \lambda\sqrt{\frac{D^+_{\max} s \log p}{n}} \right)^2}{(n-s)C_{\min}} \right).
\end{split}
\end{equation}
Then, $\Prob (\overline{\oper E} ) \geq 1 - c_1 \exp(- c_2 (n-s)\epsilon^2  \})$ for some universal constants $c_1, c_2 > 0$.
\end{lem}

Conditioned on the event $\overline{\oper E}$ defined in Lemma \ref{lem::upper bound M}, the probability $\Prob (\max_{i \in T^c} |b_i| \geq \gamma)$ is upper bounded by
\begin{align*}
\Prob (\max_{i \in T^c} |b_i| \geq \gamma \text{   } | \text{   } \overline{\oper E}) + \exp(- c_2 (n-s)) .
\end{align*}

\noindent We recall that $b_i$ is a zero-mean Gaussian random variable, thus the standard Gaussian tail bound in (\ref{lem::bound for a sub-Gaussian random variables}) allows us to derive
$$
\Prob (\max_{i \in T^c} |b_i| \geq \gamma \text{   } | \text{   } \overline{\oper E}) \leq 2(p-k) \exp \left( - \frac{\gamma^2}{2 \rho_u \overline{M}} \right).
$$

\noindent This exponential probability decays at the rate of $\exp(-c \log(p-k))$ provided that $\frac{1}{\gamma^2 }2 \rho_u \overline{M} \log(p-k)$ is strictly less than one. Now we replace the definition of $\overline{M}$ in (\ref{eqt::M overline}) into this inequality. To do this, we notice that $\frac{k}{n-s} = o(1)$ from the sample size assumption of Theorem \ref{thm::Achievability - Gaussian design}, thus we can select $\epsilon \in (0,1)$ such that $4\sqrt{\frac{k}{n-s}} \leq \epsilon$. Following some simple algebra, we find that it is sufficient to have
\begin{multline*}
\frac{n-s}{1+\epsilon} > \frac{2 \rho_u}{C_{\min} \gamma^2} k \log(p-k)\times \left\{ \frac{C_{\min} (n-s)}{(1+\epsilon) k}\frac{\lambda^2 s}{n} \right. \\
\left. + \left(1 + \lambda \sqrt{\frac{D^+_{\max} s \log p}{n}}\right)^2 + \frac{(n-s)^2}{n^2} \frac{\sigma^2 C_{\min} }{ \lambda_{n,\beta}^2 k} \right\}.
\end{multline*}

\noindent Replace the expression of $\lambda$ in (\ref{eqt::equation of lambda}) and $s = \eta n$ and perform some simple algebra, we conclude that the $\ell_{\infty}$-norm of $z^{(\beta)}_{T^c}$ is strictly less than one as long as the following bound of the sample size obeys
\begin{align*}
\frac{n}{2(1+\epsilon)} > \frac{1}{(1-\eta)}&\frac{2 \rho_u}{C_{\min} \gamma^2 } k \log(p-k) \\
&\times \left\{ \frac{9}{4} + (1-\eta)^2 \frac{\sigma^2 C_{\min} }{ \lambda_{n,\beta}^2 k }\right\},
\end{align*}
which matches with the assumption of Theorem \ref{thm::Achievability - Gaussian design}.
\end{proof}



\begin{proof} [Proof of Lemma \ref{lem::bound Linf norm of z}]
Recall the expression of $z$ in the lemma, we have by the triangular inequality, $ \norm{z}_{\infty} \leq \frac{\lambda}{\sqrt{n}} \norm{X^*_{ST} \sgn(e^{\star}_S)}_{\infty} + 1$. Furthermore, we know that the matrix $X_{ST}$ can be represented as $W_{ST} \Sigma_{TT}^{1/2}$ where $W_{ST} \in \R^{s \times k}$ is the random matrix with i.i.d. zero mean entries and unit variance. Hence,
\begin{align*}
\norm{X^*_{ ST} \sgn(e^{\star}_S)}_{\infty} &= \norm{ \Sigma^{1/2}_{TT} W^*_{S T} \sgn(e^{\star}_S)}_{\infty} \\
&\leq  \sqrt{D^+_{\max}} \norm{W^*_{S T} \sgn(e^{\star}_S)}_{\infty},
\end{align*}
where the inequality follows from matrix sub-multiplicative norm and $\norm{\Sigma^{1/2}_{TT}}_{\infty} \leq \norm{\Sigma_{TT}}^{1/2}_{\infty} = \sqrt{D^+_{\max}}$.

\noindent Consider the random variable $V_i = \inner{w_i, \sgn(e^{\star}_S)}$ where $w_i$ is a column vector of $W_{ST}$. Recall that each entry of $w_i$ is $\oper N (0, 1)$ and $\norm{\sgn(e^{\star}_S)}_2 = \sqrt{s}$. Hence, $V_i$ is a Gaussian r.v. with variance $s$. Applying Gaussian tail bound (\ref{lem::bound for a sub-Gaussian random variables}) in the Appendix together with taking the union bound yields
$$
\Prob \left( \norm{W^*_{S T} \sgn(e^{\star}_S) }_{\infty} \geq \tau \right) \leq 2 k \exp (- \tau^2 /2s ).
$$

\noindent Selecting $\tau = 2 \sqrt{s \log p}$ so that the probability exponentially decays to zero. Combining these inequalities completes the proof of Lemma \ref{lem::bound Linf norm of z}.
\end{proof}

\begin{proof} [Proof of Lemma \ref{lem::upper bound M}]
Since $\Pi_{S^c T}$ is the orthogonal projection matrix, we have $\norm{\Pi_{S^c T} w_{S^c}}_2^2 \leq \norm{w_{S^c}}_2^2$. In addition, $\frac{1}{\sigma^2} \norm{w_{S^c}}_2^2$ is the $\chi^2$-variate with $(n-s)$ degrees of freedom, thus
\begin{multline*}
\Prob \left( \frac{1}{n^2 \lambda_{n,\beta}^2} \norm{\Pi_{S^c T} w_{S^c}}_2^2 \geq (1+\epsilon) \frac{\sigma^2 (n-s)}{n^2\lambda_{n,\beta}^2} \right) \\
\leq  2 \exp \left( - \frac{3 (n-s) \epsilon^2}{16} \right).
\end{multline*}

\noindent Turning to the last term of $M$, by the spectral norm bound of the Gaussian random matrix (\ref{inq::spectral norm bound of general X^* X}), we obtain
\begin{align*}
z^* (X_{S^c T}^* X_{S^c T})^{-1} z \leq \left( 1 + 4 \sqrt{\frac{k}{n-s}} \right) \frac{\norm{z}_2^2}{(n-s)C_{\min}},
\end{align*}
with probability greater than $1 - c_1 \exp(-c_2 (n-s))$. Conditioned on the event $\oper E_z$ in Lemma \ref{lem::bound Linf norm of z}, we have $\norm{z}_2^2 \leq k \norm{z}_{\infty}^2 \leq k \left(1 + \lambda \sqrt{\frac{D^+_{\max} s \log p}{n}} \right)^2$. The proof is completed by combining these bounds.
\end{proof}

\subsection{Verify the upper bound of $\norm{z^{(e)}_{S^c}}_{\infty}$}
\label{sec::bound infty norm of z^e_S^c}

\begin{proof}
By replacing expressions of $\widehat{\beta}$ and $\widehat{e}$ into $z^{(e)}_{S^c} =  \frac{1}{\lambda_{n,e}}(y_{S^c} - X_{S^c} \widehat{\beta})$, we get
\begin{equation}
\label{eqt::final equation of z^e_S^c}
\begin{split}
z^{(e)}_{S^c} &= \frac{1}{\sqrt{n} \lambda_{n,e}} \Pi_{S^c T} w_{S^c} +  \frac{\sqrt{n}}{\lambda} X_{S^c T} ( X^*_{S^c T} X_{S^cT} )^{-1} z ,
\end{split}
\end{equation}
where we use the same notations of $\Pi_{S^c T}$ and $z$ as in the previous section: $\Pi_{S^c T} := I - X_{S^c T} ( X^*_{S^c T} X_{S^cT} )^{-1} X^*_{S^c T}$ and $z := \frac{1}{\sqrt{n}}\lambda X^*_{S T} \sgn(e^{\star}_S) - \sgn(\beta^{\star}_T)$. To show that $\norm{z^{(e)}_{S^c}}_{\infty} < 1$, we bound $\ell_{\infty}$-norm of each term of the sum (\ref{eqt::final equation of z^e_S^c}) separately. In particular, we will establish that with probability converging to one, the $\ell_{\infty}$-norm of the first term is bounded by $\frac{2 \sigma \sqrt{\log n}}{\lambda_{n,e} \sqrt{n}}$ and that of the second term is less than $(1 - \frac{2 \sigma \sqrt{\log n}}{\lambda_{n,e} \sqrt{n}})$. The proof is therefore completed by the triangular inequality.

We begin by establishing the $\ell_{\infty}$-norm of the first term of $z^{(e)}_{S^c}$ in (\ref{eqt::final equation of z^e_S^c}):
$$
\frac{1}{\sqrt{n} \lambda_{n,e}} \norm{\Pi_{S^c T} w_{S^c}}_{\infty} = \max_{i}  \frac{1}{ \sqrt{n} \lambda_{n,e}}  \left|\inner{ u_i , w_{S^c} }  \right|,
$$
where $u_i$ is a column vector of $\Pi_{S^c T}$. Since $\frac{1}{ \sqrt{n}\lambda_{n,e}} \inner{ u_i, w_{S^c} }$ is a sum of Gaussian random variables with zero mean and variance $\frac{\sigma^2}{n\lambda_{n,e}^2} \norm{u_i}_2^2$, it can be bounded by the Gaussian tail inequality in (\ref{lem::bound for a sub-Gaussian random variables}) in Appendix \ref{appen::concentration inequalities}. Notice that spectral norm of any orthogonal projection is one, $\norm{u_i}_2 \leq 1$. We have
$$
\Prob \left( \frac{1}{ \sqrt{n} \lambda_{n,e}}  \left| \inner{ u_i, w_{S^c} }  \right|  \geq \tau \right)  \leq 2 \exp  \left( - \frac{n \lambda^2_e \tau^2 }{2 \sigma^2} \right).
$$

\noindent Choose $\tau = \frac{2 \sigma \sqrt{\log n}}{\sqrt{n} \lambda_{n,e}}$ and take the union bound over all $|S^c|$ columns of the matrix $\Pi_{S^c T}$, we have
\begin{multline}
\label{inq::bound l-infty PiScT wSc}
\Prob \left( \frac{1}{\sqrt{n} \lambda_{n,e}} \norm{ \Pi_{S^c T} w_{S^c}}_{\infty} \geq \frac{2 \sigma \sqrt{\log n}}{\sqrt{n}\lambda_{n,e}} \right) \\
\leq 2 |S^c| \exp  \left( - 2 \log n \right).
\end{multline}


Next, we controlthe upper bound of $\frac{\sqrt{n}}{\lambda}\norm{X_{S^c T} (X^*_{S^c T} X_{S^c T})^{-1} z}_{\infty}$. The following lemma, whose proof is deferred to Appendix \ref{app::proof bound inf norm of Xsct inv Xsct Xsct z}, establishes this bound.
\begin{lem}
\label{lem::bound inf norm of Xsct inv Xsct Xsct z}
Under the assumptions of Theorem \ref{thm::Achievability - Gaussian design}, for any vector $z \in \R^k$ independent with $X_{S^c T}$, the following statement holds
$$
\frac{\sqrt{n}}{\lambda}\norm{X_{S^c T} (X_{S^c T}^* X_{S^c T} )^{-1} z }_{\infty} < \frac{2}{3} \left( 1 - \frac{2\sigma \sqrt{\log n}}{\lambda_{n,e} \sqrt{n}} \right) \norm{z}_{\infty}
$$
with probability greater than $1 - c_1 \exp(- c_2\max \{\log (p-k), \log(n-s) \})$.
\end{lem}

Since $\sgn(\beta^{\star}_T)$ and $X^*_{ST} \sgn(e^{\star}_S)$ are statistically independent with $X_{S^c T}$, $z := \frac{1}{\sqrt{n}}\lambda X^*_{S T} \sgn(e^{\star}_S) - \sgn(\beta^{\star}_T)$ satisfies the assumption of Lemma \ref{lem::bound inf norm of Xsct inv Xsct Xsct z}. Moreover, by Lemma \ref{lem::bound Linf norm of z} and the definition of $\lambda$ in (\ref{eqt::equation of lambda}), we have with high probability
$$
\norm{z}_{\infty} \leq 1 + \lambda \sqrt{\frac{D^+_{\max} s \log p}{n}} \leq \frac{3}{2},
$$
where the last inequality holds from the assumption of Theorem \ref{thm::Achievability - Gaussian design}. Now, applying Lemma \ref{lem::bound inf norm of Xsct inv Xsct Xsct z} leads to $\frac{\sqrt{n}}{\lambda}\norm{X_{S^c T} (X^*_{S^c T} X_{S^c T})^{-1} z}_{\infty} \leq 1 - \frac{2\sigma \sqrt{\log n}}{\lambda_{n,e} \sqrt{n}}$.

Putting these two bounds together and using the triangular inequality we conclude that with high probability, $\norm{z^{(e)}_{S^c}}_{\infty} < 1$ as claimed.
\end{proof}

\subsection{Establish the $\ell_{\infty}$ bound of $\widehat{\beta}_T - \beta^{\star}_T$}
\label{sec::upper Linf h_T}

Recall the formula of $(\widehat{\beta}_T - \beta^{\star}_T)$ from (\ref{eqt::h_T}), the triangular inequality yields
\begin{equation}
\begin{split}
&\norm{\widehat{\beta}_T - \beta^{\star}_T}_{\infty} \leq \norm{  \left( X^*_{S^c T} X_{S^c T} \right)^{-1} X^*_{S^c T} w_{S^c}}_{\infty} \\
&{ }+ n \lambda_{n,\beta} \norm{(X^*_{S^c T} X_{S^c T})^{-1} (\frac{1}{\sqrt{n}}\lambda X^*_{ST} \sgn(e^{\star}_S) - \sgn(\beta^{\star}_T) )}_{\infty} \\
&:= \oper T_1 + \oper T_2.
\end{split}
\end{equation}

To bound the first quantity, we consider a random vector $u =  (\frac{1}{n-s}X^*_{S^c T} X_{S^c T})^{-1} \frac{1}{n-s}X^*_{S^c T} w_{S^c}$ and note that $\oper T_1 = \norm{u}_{\infty}$. This bound, which is stated below, has been established in equation (42) of \cite{Wainwright_Lasso_2009_J}: there exists some numerical constant $c$ such that
\begin{equation}
\label{eqt::bound T1 in beta_t - beta*T}
\Prob \left( \oper T_1 \geq 20 \sqrt{\frac{\sigma^2 \log k}{C_{\min} (n-s)} } \right) \leq 4 \exp (- c (n-s)).
\end{equation}

Turning now to the second quantity $\oper T_2$. We have
$$
\oper T_2 \leq \frac{n \lambda_{n,\beta}}{n-s} \norm{\left(\frac{ X^*_{S^c T} X_{S^c T} }{n-s} \right)^{-1} z }_{\infty},
$$
where $z := \frac{1}{\sqrt{n}}\lambda X^*_{ST} \sgn(e^{\star}_S) - \sgn(\beta^{\star}_T)$. To bound $\oper T_2$, we follow similar arguments in \cite{Wainwright_Lasso_2009_J}, Section V.B. We can now state the following lemma, which is modified from Lemma 5 of \cite{Wainwright_Lasso_2009_J}.
\begin{lem}
\label{lem::support lemma for Linf bound of hT}
Let $z \in \R^k$ be a fixed nonzero vector and $W \in \R^{n \times k}$ be a random matrix with i.i.d entries $W_{ij} \sim \oper N(0,1)$. Then, there exists positive constants $c_1$ and $c_2$ such that
\begin{equation*}
\begin{split}
\Prob &\left( \norm{ \left[ \left( \frac{W^* W}{n} \right)^{-1} - I_{k\times k} \right] z }_{\infty} \geq c_1 \sqrt{\frac{k \log (p-k)}{n}} \norm{z}_{\infty} \right) \\
&\leq 4 \exp(-c_2 \min \{k, \log(p-k) \}).
\end{split}
\end{equation*}
\end{lem}

\noindent Following similar arguments as in \cite{Wainwright_Lasso_2009_J}, Section V.B, we have a similar probabilistic bound as equation (41) of \cite{Wainwright_Lasso_2009_J}
\begin{multline}
\label{eqt::bound T2 - step 1}
\Prob \left( \oper T_2 \geq c_1 \lambda_{n,\beta} \sqrt{\frac{k n \log (p-k)}{(n-s)^2 }} \norm{\Sigma^{-1/2}_{TT}}_{\infty} \norm{\Sigma^{-1/2}_{TT} z}_{\infty} \right) \\
\leq 4 \exp(-c_2 \min \{k, \log(p-k) \}).
\end{multline}

\noindent Furthermore, Lemma \ref{lem::bound Linf norm of z} states that $\norm{z}_{\infty} \leq 3/2$ with high probability. Conditioning on the event $\oper E = \{ \norm{z}_{\infty} \leq 3/2 \}$, we have $\norm{\Sigma^{-1/2}_{TT} z}_{\infty} \leq \frac{3}{2}\norm{\Sigma^{-1/2}_{TT}}_{\infty}$. Thus, (\ref{eqt::bound T2 - step 1}) leads to
\begin{multline*}
\Prob \left( \oper T_2 \geq c_2 \lambda_{n,\beta} \sqrt{\frac{k n \log (p-k)}{(n-s)^2}} \norm{\Sigma^{-1/2}_{TT}}^2_{\infty} \text{  }| \text{  } \oper E \right) \\
\leq 4 \exp(-c_2 \min \{k, \log(p-k) \}).
\end{multline*}

\noindent By the total probability rule, $\Prob(\oper T_2 \geq \tau) \leq \Prob(\oper T_2 | \oper E) + \Prob(\oper E^c)$. Therefore, we conclude that with probability greater than $1 - 6 \exp(-c_2 \min \{k, \log(p-k) \})$,
\begin{equation}
\label{eqt::bound T2 - step 2}
\oper T_2 \leq c_2 \lambda_{n,\beta} \sqrt{\frac{k \log (p-k)}{(1-\eta)^2 n}} \norm{\Sigma^{-1/2}_{TT}}^2_{\infty}.
\end{equation}

Overall, combining the bound of $\oper T_2$ with the bound of $\oper T_1$ in (\ref{eqt::bound T1 in beta_t - beta*T}) concludes that $\norm{\widehat{\beta}_T - \beta^{\star}_T}_{\infty} \leq f_{\beta} (\lambda_{n,\beta})$ with probability at least $1 - 10 \exp(-c_3 \min \{k, \log(p-k) \}$ where $f_{\beta} (\lambda_{n,\beta})$ is defined in (\ref{eqt::f-beta}).

\subsection{Establish the $\ell_{\infty}$ bound of $\widehat{e}_S - e^{\star}_S$}
\label{sec::upper Linf f_S}

Recalling the formula of $\widehat{e}_S - e^{\star}_S$ in (\ref{eqt::g_S}) and applying the triangular inequality, we get
\begin{equation}
\label{eqt::KKT condition on e_S - e*_S - final step}
\begin{split}
\norm{\widehat{e}_S - e^{\star}_S}_{\infty} &\leq \frac{1}{\sqrt{n}} \norm{ X_{ST} ( X^*_{S^c T} X_{S^cT} )^{-1} X^*_{S^c T} w_{S^c}}_{\infty} \\
&{ }+ \lambda_{n,\beta} \sqrt{n} \norm{ X_{ST} ( X^*_{S^c T} X_{S^cT} )^{-1} z  }_{\infty}\\
&{ }+ \frac{1}{\sqrt{n}} \norm{w_S}_{\infty} + \lambda_{n,e} \\
&:= \oper T_1 + \oper T_2 + \oper T_3 + \lambda_{n,e},
\end{split}
\end{equation}
where we again denote $z = \frac{1}{\sqrt{n}}\lambda X^*_{S T} \sgn(e^{\star}_S) - \sgn(\beta^{\star}_T)$. We first consider the easiest term $\oper T_3 = \frac{1}{\sqrt{n}}\norm{w_S}_{\infty}$. Since $w_S$ is a random vector with i.i.d. $\oper N (0, \sigma^2)$ entries, by Gaussian extreme order statistics \cite{LT_1991_B}, $\oper T_3 \leq 2 \sqrt{\frac{\sigma ^2 \log s}{n}}$.

Turning to the first term $\oper T_1$, we define a vector $v \in \R^s$ whose entries are $v_i := x_i ( X^*_{S^c T} X_{S^cT} )^{-1} X^*_{S^c T} w_{S^c}$ where $x_i$ is the $i$-th row of the matrix $X_{ST}$ and notice that $\oper T_1 = \norm{v}_{\infty}$. Conditioned on $X_T$, it is clear that $v_i$ is a zero mean random variable with variance $\sigma^2 x_i ( X^*_{S^c T} X_{S^cT} )^{-1} x^*_i$. In addition, we recall that $X_T$ can be represented as $X_T = W_T \Sigma^{1/2}_{TT}$ where $W_T $ is the $n \times k$ standard Gaussian matrix. Thus, $x_i ( X^*_{S^c T} X_{S^cT} )^{-1} x^*_i = w_i ( W^*_{S^c T} W_{S^cT} )^{-1} w^*_i \leq \norm{w_i}_2^2 \norm{( W^*_{S^c T} W_{S^cT} )^{-1}}$, where $w_i$ is the $i$-th row of matrix $W_{ST}$. In short, $v_i$ is a zero mean random variable with variance at most $ \widetilde{\sigma}^2 := \sigma^2 \norm{w_i}_2^2 \norm{( W^*_{S^c T} W_{S^cT} )^{-1}}$. Using the concentration result for $\chi^2$-variate, we get $\norm{w_i}_2^2 \leq 2k$ with probability at least $1 - \exp(-k/2)$. Furthermore, from random matrix theory (\ref{inq::spectral norm bound of inverse of X^* X}) in Appendix \ref{appen::concentration inequalities}, $\norm{( W^*_{S^c T} W_{S^cT} )^{-1}} \leq \frac{5}{n-s}$ with probability at least $1 - \exp(-(n-s)/2)$.

Next, let us define the event
$$
\oper E = \left\{ \widetilde{\sigma}^2 \geq \frac{10 \sigma^2 k}{n-s} \right\}.
$$

\noindent From the above arguments, we have $\Prob (\oper E) \leq \exp(- (n-s +k)/2 )$. By the total probability rule, we have
$$
\Prob (\oper T_1 \geq \tau) \leq \Prob (\oper T_1 \geq \tau | \oper E^c) + \Prob (\oper E).
$$

\noindent Conditioning on $\oper E^c$, $v_i$ is zero mean Gaussian with variance at most $\frac{10 \sigma^2 k}{n-s} $. Thus, by the Gaussian tail bound (\ref{lem::bound for a sub-Gaussian random variables}) in Appendix \ref{appen::concentration inequalities}, we derive
$$
\Prob \left( \max_{i \in S} |v_i| \geq \tau \right) \leq 2 s \exp \left( - \frac{(n-s) \tau^2}{10 \sigma^2 k} \right)
$$

\noindent Setting $\tau = \sqrt{\frac{20 \sigma^2 k \log (p-k) }{n-s}}$ yields the fact that this probability vanishes at rate $2 (p-k)^{-1}$. Overall, we can now conclude that
\begin{equation*}
\Prob \left( \oper T_1 \geq 11 \sqrt{\frac{\sigma^2 k \log (p-k) }{n-s}} \right) \leq 2 \exp(- \log (p-k)).
\end{equation*}

It is left to bound $\oper T_2$. By sub-multiplicative norm inequality, $\oper T_2$ is bounded by
$$
\lambda_{n,\beta} \sqrt{n} \norm{X_{ST}}_{\infty} \norm{( X^*_{S^c T} X_{S^cT} )^{-1} z}_{\infty},
$$
We already established $n \lambda_{n,\beta}\norm{( X^*_{S^c T} X_{S^cT} )^{-1} z}_{\infty}$ in (\ref{eqt::bound T2 - step 2}). In addition, $\norm{X_{ST}}_{\infty} \leq \sqrt{k} \norm{X_{ST}}$ where by the matrix theory (\ref{inq::spectral norm bound of general X^* X}) in Appendix \ref{appen::concentration inequalities}, $\norm{X^*_{ST} X_{ST}} \leq 4 C_{\max} (s + \sqrt{sk})$ with high probability. Thus, $\norm{X_{ST}}_{\infty} \leq \sqrt{C_{\max}} (sk + k \sqrt{sk})^{1/2}$.

Overall, combining with the bounds of $\oper T_1$ and $\oper T_3$, we conclude that $\norm{\widehat{\beta}_T - \beta^{\star}_T}_{\infty} \leq f_{\beta} (\lambda_{n,\beta})$ with probability at least $1 - 10 \exp(-c_3 \min \{k, \log(p-k) \}$ where $f_e (\lambda_{n,\beta}, \lambda_{n,e})$ is defined as in (\ref{eqt::f-e}).

\section{Proof of Theorem \ref{thm::Inachievability - Gaussian design} - Inachievability}
\label{sec::proof 3}


Our analysis in this section relies on the the notion of primal-dual witness introduced by Wainwright \cite{Wainwright_Lasso_2009_J}. In particular, we will construct a pair of primal solutions $(\widehat{\beta}, \widehat{e})$ and their dual vectors $(z^{(\beta)}, z^{(e)})$. The extended Lasso (\ref{opt::lasso with sparse noise}) fails to correctly identify signed support of the coefficient vector $\beta^{\star}$ and the error $e^{\star}$ when the $\ell_{\infty}$-norm of either $z^{(\beta)}_{T^c}$ or $z^{(e)}_{S^c}$ exceeds unity with probability approaching one. The primal-dual witness is constructed as follows:

\begin{enumerate}
    \item First, we obtain the solution pair $(\widehat{\beta}_T, \widehat{e}_S)$ of the following restricted Lasso problem
        \begin{equation}
        \label{opt::restricted extended Lasso}
        \min_{\beta, e} \frac{1}{2n} \norm{y_S - X_{ST} \beta_T - \sqrt{n} e_S}_2^2 + \lambda_{n,\beta} \norm{\beta_T}_1 + \lambda_{n,e} \norm{e_S}_1.
        \end{equation}
        We also set $\widehat{\beta}_{T^c} = 0$ and $\widehat{e}_{S^c} = 0$.
    \item Second, we select $z^{(\beta)}_T$ and $z^{(e)}_S$ as elements of the subgradients $\norm{\widehat{\beta}}_1$ and $\norm{\widehat{e}}_1$, respectively.

    \item Third, we solve for vectors $z^{(\beta)}_{T^c}$ and $z^{(e)}_{S^c}$ satisfying the KKT conditions in (\ref{eqt::KKT condition on beta}). We then verify whether the dual feasibility conditions of both $\norm{z^{(\beta)}_{T^c}}_{\infty} < 1$ and $\norm{\widehat{e}_{S^c}}_{\infty} < 1$ are satisfied.
    \item Fourth, we check whether the sign consistency $z^{(\beta)}_T = \sgn(\beta^{\star}_T)$ and $z^{(e)}_S = \sgn(e^{\star}_S)$ are satisfied.
\end{enumerate}

The following result summarizes the use of the primal-dual witness construction in providing the proof of Theorem \ref{thm::Inachievability - Gaussian design}:
\begin{lem}
If either steps 3 or 4 of the primal-dual construction fails, then the extended Lasso fails to recover the correct signed supports of both $\beta^{\star}$ and $e^{\star}$.
\end{lem}
The proof of this lemma is essentially similar to that of Lemma 2(c) in \cite{Wainwright_Lasso_2009_J}, thus we omit the detail here.

In our proof, we assume that $z^{(\beta)}_T = \sgn(\beta^{\star}_T)$ and $z^{(e)}_S = \sgn(e^{\star}_S)$; otherwise, the sign consistency would fails. Under these assumptions, it is easy to check that the solution $(\widehat{\beta}_T, \widehat{e}_S)$ of the optimization (\ref{opt::restricted extended Lasso}) is expressed in (\ref{eqt::h_T}) and (\ref{eqt::g_S}). Thus, we can derive equations of $z^{(\beta)}_{T^c}$ and $z^{(e)}_{S^c}$ as in (\ref{eqt::final equation of z^beta_T^c}) and (\ref{eqt::final equation of z^e_S^c}).

In the following two sections, we establish the claim by showing that under the conditions of the sample size $n$ and $s = \eta n$ as in Theorem \ref{thm::Inachievability - Gaussian design}, the $\ell_{\infty}$-norm of either $z^{(\beta)}_{T^c}$ or $z^{(e)}_{S^c}$ exceeds unity with probability tending to one. It is clear that if the extended Lasso (\ref{opt::lasso with sparse noise}) fails to recover signed support vectors with $s = \eta n$, it also fails to do so with $s > \eta n$ since it is easier to solve the extended Lasso when there is less corrupted observations.

\subsection{Lower $\ell_{\infty}$-norm bound of $z^{(\beta)}_{T^c}$}

Recall the expression of $z^{(\beta)}_{T^c}$ in (\ref{eqt::final equation of z^beta_T^c}) and its simplified form $z^{(\beta)}_{T^c} = a + b$ where $b$ and $a$ are defined in (\ref{eqt::define b}) and (\ref{eqt::define a}). We already have $\norm{a}_{\infty} \leq 1-\gamma $ due to the mutual incoherence assumption. It is now sufficient to show that $\max_{i \in T^c} |b_i|$ exceeds $(2-\gamma)$ with high probability.

Conditioning on $X_{T}$ and $w$, the vector $b$ is zero-mean Gaussian with covariance matrix $M \Sigma_{T^c|T}$ where the random scaling form $M$ has the form (\ref{eqt::M}). The following lemma controls the lower bound of this scaling factor. The proof is similar to that of Lemma 6 in \cite{Wainwright_Lasso_2009_J}, so we omit the detail here.
\begin{lem}
Define the event $\oper E = \{ M > \underline{M} \}$, where $\underline{M}$ is defined in (\ref{eqt::M underline}). Then, $\Prob (\oper E) \leq 1 - c_1 \exp(- c_2 (n-s))$ for some $c_1, c_2 > 0$.
\end{lem}

\newcounter{tempequationcounter}
\begin{figure*}[!t]
\normalsize
\setcounter{tempequationcounter}{\value{equation}}
\begin{equation}
\label{eqt::M underline}
\underline{M} :=
\begin{cases}
\frac{\lambda^2 s}{n} + c \frac{k}{n-s} \quad &\text{if } k/n = \Theta(1) \\
\frac{\lambda^2 s}{n} + \left( 1 - \max \left\{ \epsilon, 4 \sqrt{\frac{k}{n-s}} \right\} \right)
\left( \frac{\sigma^2 (n-s)}{n^2 \lambda_{n,\beta}^2} + \frac{\norm{z}_2^2}{(n-s)C_{\max}} \right) \quad &\text{if } k/n = o(1).
\end{cases}
\end{equation}
\hrulefill
\vspace*{4pt}
\end{figure*}

Following the proof of Theorem 4 in \cite{Wainwright_Lasso_2009_J}, we have the following lower bound: for all $\nu, \epsilon, \tau > 0$
\begin{equation}
\label{inq::first bound of max b_i}
\max_{i \in T^c} |b_i| \geq \sqrt{(2-\nu) \rho_l(\Sigma_{T^c|T}) \underline{M} \log (p-k) } - \tau
\end{equation}
with probability at least $1 - 2 \exp \left( -\frac{\tau^2}{2 \underline{M} \rho_u} \right)$. Now, using appropriate choices of $\{\tau, \nu, \gamma \}$, it suffices to establish the bound
\begin{equation}
\label{inq::lower bound underlineM}
\rho_l(\Sigma_{T^c|T})  \underline{M} \log(p-k) \geq \frac{[(2-\gamma) + \tau]^2}{(2-\nu)}.
\end{equation}

\noindent We consider two cases:

1) If $\underline{M} \rightarrow + \infty$ or $\underline{M} = \Theta(1)$, then we can choose $\tau^2 = \delta \underline{M} \log(p-k)$ for some $\delta > 0$. For $\delta$ sufficiently small, we conclude from (\ref{inq::first bound of max b_i}) that with probability converging to one, there exists some constants $c >0$ such that
$$
\max_{i \in T^c} |b_i| \geq c\sqrt{\log(p-k)},
$$
which exceeds $(2-\gamma)$ regardless of the choice of the sample size $n$.


2) Otherwise, $\underline{M} = o(1)$. This is satisfied only if $k/n = o(1)$ and thus, the second line of the definition of $\underline{M}$ is applied. Now, we can select $\tau$ sufficiently small and have a guarantee that $\frac{\tau^2}{2 \underline{M}} \rightarrow +\infty$. From the definition of $\underline{M}$, one can see that if $\rho_l \frac{\lambda^2 s}{n} \log(p-k) \geq 2$, we can choose $\tau$ and $\nu$ strictly positive but arbitrarily close to zero such that $\frac{[(2-\gamma) + \tau]^2}{(2-\nu)} < 2$. Thus, (\ref{inq::lower bound underlineM}) obeys regardless of the selection of the sample size $n$. Consequently, we assume that
\begin{equation}
\label{inq::upper bound of lambda}
\lambda < \sqrt{\frac{2n}{\rho_l s \log(p-k)}}.
\end{equation}

\noindent Under this assumption, we can lower bound $\norm{z}_2$ as follows
\begin{equation}
\begin{split}
\label{inq::boudn L2 of z}
\norm{z}_2 &= \norm{\sgn(\beta^\star_T) - \frac{1}{\sqrt{n}}\lambda X^*_{ST} \sgn(e^\star_S) }_2 \\
&\geq \norm{\sgn(\beta^\star_T)}_2 - \frac{1}{\sqrt{n}}\lambda \norm{X^*_{ST} \sgn(e^\star_S)}_2 \\
&\geq \sqrt{k} - \lambda \sqrt{\frac{k}{n}} \norm{X^*_{ST} \sgn(e^\star_S)}_{\infty}.
\end{split}
\end{equation}

\noindent As shown during the proof of Lemma \ref{lem::bound Linf norm of z} that $\frac{1}{\sqrt{n}}\norm{X^*_{ST} \sgn(e^\star_S)}_{\infty} \leq \frac{1}{3} \sqrt{\frac{ \rho_l s \log (p-k)}{n}}$ with probability greater than $1 - \exp(-\frac{\rho_l}{18D^+_{\max}} \log p)$, from the above upper bound of $\lambda$, we obtain  $\frac{\lambda}{\sqrt{n}}\norm{X^*_{ST} \sgn(e^\star_S)}_{\infty} \leq \frac{\sqrt{2}}{3}$. Consequently, we achieve the lower bound with high probability
\begin{equation}
\label{inq::lower bound of L2 of z}
\norm{z}_2 \geq \frac{1}{2} \sqrt{k}.
\end{equation}

\noindent Furthermore, for $(n-s)$ sufficiently large, we select a $\epsilon \in (0,1/2)$ such that $4 \sqrt{\frac{k}{n-s}} < \epsilon$. Now, replace this bound into the second equation of $\underline{M}$ and perform some simple algebra, we can show that the inequality (\ref{inq::lower bound underlineM}) is satisfied as long as
\begin{multline*}
\frac{\rho_l}{C_{\max}} \frac{k \log(p-k)}{(n-s)} \left\{ \frac{C_{\max} \lambda^2 s (n-s)}{(1-\epsilon) k n}  \right. \\
\left. + \frac{1}{4} + \frac{(n-s)^2}{n^2} \frac{\sigma^2 C_{\max}}{\lambda_{n,\beta}^2 k}  \right\}
\geq \frac{[(2-\gamma) + \tau]^2}{(2-\nu) (1-\epsilon)}.
\end{multline*}

\noindent Replace the lower bound of $\lambda$ in (\ref{inq::lower bound of lambda}) and $s = \eta n$ into the above inequality, we can conclude that the inequality (\ref{inq::lower bound underlineM}) is satisfied as long as
\begin{align*}
\frac{\rho_l}{C_{\max}(2-\gamma)^2} &\frac{2 k \log(p-k)}{(n-s)} \left\{\frac{3}{8} + (1-\eta)^2 \frac{\sigma^2 C_{\max}}{\lambda_{n,\beta}^2 k}  \right\} \\
&\geq \frac{[(2-\gamma) + \tau]^2}{(2-\gamma)^2(1-\nu/2) (1-\epsilon)}.
\end{align*}

\noindent Under the assumptions of Theorem \ref{thm::Inachievability - Gaussian design}, the right-hand side is strictly greater than one. On the other hand, $\tau, \nu$ and $\epsilon$ are parameters that can be chosen in $(0,1/2)$. By selecting these parameters to be positive but arbitrarily close to zeros, we can set the right-hand side less than one. Therefore, (\ref{inq::lower bound underlineM}) is satisfied.

\subsection{Lower the $\ell_{\infty}$-norm bound of $z^{(e)}_{S^c}$}

Recalling the equation of $z^{(e)}_{S^c}$ in (\ref{eqt::final equation of z^e_S^c}), we have
$$
z^{(e)}_{S^c} = \frac{1}{ \sqrt{n} \lambda_{n,e}} \Pi_{S^c T} w_{S^c} +  \frac{\sqrt{n}}{\lambda} X_{S^c T} ( X^*_{S^c T} X_{S^cT} )^{-1} z
$$
where we recall $z = \frac{1}{\sqrt{n}} \lambda X^*_{ST} \sgn(e^{\star}_S) - \sgn(\beta^{\star}_T)$. First, notice that $\Pi_{S^c T}$ is the orthogonal projection onto the column space of the matrix $X_{S^c T}$. Thus, two terms in the above summation are orthogonal to each other. Therefore, lowering the $\ell_{\infty}$-norm of $z^{(e)}_{S^c}$ by its $\ell_2$-norm counterpart, we have
\begin{align*}
&(n-s) \norm{z^{(e)}_{S^c}}^2_{\infty} \geq  \norm{z^{(e)}_{S^c}}_2^2 \\
&= \frac{1}{n \lambda_{n,e}^2} \norm{\Pi_{S^c T} w_{S^c}}_2^2 + \frac{n}{\lambda^2} \norm{X_{S^c T} ( X^*_{S^c T} X_{S^cT} )^{-1} z}_2^2.
\end{align*}

From this inequality, we have an important observation that both terms in the sum have to be upper bounded by $(n-s)$. Otherwise, $\norm{z^{(e)}_{S^c}}^2_{\infty}$ is automatically strictly greater than one, regardless of the choice of the sample size $n$. This observation suggests to us the required lower bound of $\lambda_{n,e}$ and $\lambda$:
$$
\lambda_{n,e} \geq \frac{1}{\sqrt{n(n-s)}} \norm{\Pi_{S^c T} w_{S^c}}_2,
$$
and
$$
\lambda \geq \sqrt{\frac{n}{n-s}} \norm{X_{S^c T} ( X^*_{S^c T} X_{S^cT} )^{-1} z}_2.
$$
\noindent We now explicitly establish the lower bound of these regularization parameters. First, since $\frac{1}{\sigma^2} \norm{\Pi_{S^c T} w_{S^c}}_2^2$ is the $\chi^2$-variate with $n-s$ degrees of freedom, Lemma \ref{lem::chi square bound} in Appendix \ref{appen::concentration inequalities} suggests to us that $\frac{1}{\sigma^2} \norm{\Pi_{S^c T} w_{S^c}}_2^2 \geq \frac{1}{2}(n-s)$ with probability at least $1 - \exp(-(n-s)/16)$. Consequently, we require
\begin{equation}
\label{inq::lower bound of lambda_e}
 \lambda_{n,e} \geq \sqrt{\frac{\sigma^2}{2n}}.
\end{equation}

\noindent Furthermore, we observe that with probability converging to one
\begin{align*}
\norm{X_{S^c T} ( X^*_{S^c T} X_{S^cT} )^{-1} z}_2^2 &= z^* ( X^*_{S^c T} X_{S^c T} )^{-1} z \\
&= z^* \Sigma^{-1/2}_{TT} (W^*_{S^cT} W_{S^cT})^{-1} \Sigma^{-1/2}_{TT} z  \\
&\geq \norm{\Sigma^{-1/2}_{TT} z}^2_2 \sigma_{\min} ((W^*_{S^cT} W_{S^cT})^{-1}) \\
&\geq \frac{1}{2n} C^{-1}_{\max} \norm{z}_2^2,
\end{align*}
where the second identity follows from the decomposition $X_{S^cT} = \Sigma^{1/2}_{TT} W_{S^cT}$ and the last inequality is due to the Gaussian random matrix inequality (\ref{inq::bound singular values of inverse matrix of X^*X}) in Appendix \ref{appen::concentration inequalities}. In combination with the lower bound of $\norm{z}_2$, we require
\begin{equation}
\label{inq::lower bound of lambda}
\lambda \geq \sqrt{\frac{k}{8C_{\max}(n-s)}}.
\end{equation}

Turning to establish the lower bound of $\norm{z^{(e)}_{S^c}}_{\infty}$, we can show that under the assumptions of Theorem \ref{thm::Inachievability - Gaussian design}, this quantity is strictly greater than one. By the triangular inequality, $\norm{z^{(e)}_{S^c}}_{\infty} \geq \oper T_1 - \oper T_2 $ where $\oper T_1$ is quantified as
$$
\frac{\sqrt{n}}{\lambda} \norm{X_{S^c T} (X^*_{S^c T} X_{S^c T})^{-1} z}_{\infty}
$$
and the other term is $\oper T_2 := \frac{1}{\lambda_{n,e} \sqrt{n}} \norm{\Pi_{S^c T} w_{S^c}}_{\infty}$. As shown at the beginning of Section \ref{sec::bound infty norm of z^e_S^c}, we have the following inequality to hold with probability greater than $1 - 2 \exp(-\log(n-s))$:
\begin{equation*}
\oper T_2 \leq \frac{2\sqrt{\sigma^2 \log (n-s)}}{\lambda_{n,e} \sqrt{n}}.
\end{equation*}

It is now left to justify that under the assumption of Theorem \ref{thm::Inachievability - Gaussian design}, $\oper T_1 > 1 + \frac{2\sqrt{\sigma^2 \log n}}{\lambda_{n,e} \sqrt{n}}$. The remainder of this section is devoted to establish this claim. In what follows, we state two important lemmas, which are the main factor in establishing the lower bound of $\oper T_1$. The proofs of these lemmas are again deferred to the Appendix.

\begin{lem}
\label{lem::inachievability - main support lemma}
For any vector $z \in \R^k$ independent with $X_{S^cT}$, we have with probability greater than $1 - \exp(-\log(n-s))$
\begin{multline*}
\norm{ X_{S^cT} (X^*_{S^cT} X_{S^cT})^{-1} z - \frac{1}{n-s} X_{S^c T} \Sigma^{-1}_{TT} z }_{\infty} \\
\leq 16 \frac{\norm{z}_2 \sqrt{2 k \log (n-s)}}{\sqrt{C_{\min} (n-s)^3}}.
\end{multline*}
\end{lem}

\begin{lem}
\label{lem::lower l-infty bound of Xsct z}
With probability at least $1 - 4 \exp(- \frac{1}{4}\log (n-s))$,
\begin{equation*}
\norm{X_{S^c  T} \Sigma^{-1}_{TT} z}_{\infty} \geq  \frac{2\norm{z}_2 \sqrt{\log(n-s)}}{3 \sqrt{C_{\max}}}.
\end{equation*}
\end{lem}

Once these two lemmas are established, we can now show that under the assumptions of Theorem \ref{thm::Inachievability - Gaussian design}, $\oper T_1 > 1 + \frac{\sqrt{\sigma^2 \log n}}{\lambda_{n,e} \sqrt{n}}$ with high probability. By definition, $z = \frac{1}{\sqrt{n}} \lambda X^*_{ST} \sgn(e^{\star}_S) - \sgn(\beta_{T}^{\star}) $, one can see that $z$ is independent from $X_{S^cT}$. Thus, by Lemmas \ref{lem::inachievability - main support lemma} and \ref{lem::lower l-infty bound of Xsct z} and the triangular inequality, we have, with probability at least $1 - \exp(-\log(p-k)) - 4\exp(-\frac{1}{4}\log(n-s))$,
\begin{equation}
\label{eqt::bound T1}
\begin{split}
&\norm{ X_{S^cT} (X^*_{S^cT} X_{S^cT})^{-1} z}_{\infty} \\
&\geq \frac{1}{n-s} \norm{X_{S^c T} \Sigma^{-1}_{TT} z}_{\infty} - 16 (1+\epsilon) \frac{\sqrt{k \log (n-s)}}{\sqrt{C_{\min} (n-s)^3}} \norm{z}_2 \\
&\geq \left( \frac{2\sqrt{\log (n-s)}}{3(n-s)\sqrt{C_{\max}}} \right. \\
&\left. - \frac{\sqrt{\log (n-s)}}{(n-s) \sqrt{C_{\max}}} \sqrt{\frac{256 (1+\epsilon)^2k C_{\max}}{(n-s) C_{\min}}} \right) \norm{z}_2.
\end{split}
\end{equation}

\noindent Recall from the previous section that we require the upper bound of $\lambda$ in (\ref{inq::upper bound of lambda}). Otherwise, $\norm{z^{(\beta)}_{T^c}}_{\infty}$ is strictly greater than one regardless of the choice of the sample size $n$. This upper bound of $\lambda$ leads to the lower bound of $\norm{z}_2$ in (\ref{inq::lower bound of L2 of z}). Furthermore, assuming that $n-s \geq c \frac{C_{\max}}{C_{\min}} k$ for some large enough constant $c$, we achieve
$$
\norm{ X_{S^cT} (X^*_{S^cT} X_{S^cT})^{-1} z}_{\infty} \geq \frac{1}{6} (1-\epsilon) \frac{\sqrt{k \log (n-s)}}{(n-s)\sqrt{C_{\max}}}.
$$

\noindent Therefore, the requirement $\oper T_1 > 1 + \frac{2\sqrt{\sigma^2 \log n}}{\lambda_{n,e} \sqrt{n}}$ is equivalent to
$$
(n-s)^2 < \frac{(1-\epsilon)}{6} \left( 1 + \frac{2\sqrt{\sigma^2 \log n}}{\lambda_{n,e} \sqrt{n}} \right)^{-2} \frac{k n \log(n-s)}{\lambda^2 C_{\max}}.
$$

\noindent Replace the upper bound of $\lambda$ in (\ref{inq::upper bound of lambda}) and $s = \eta n$, the above inequality, or equivalently, $\norm{z^{(e)}_{S^c}}_{\infty} > 1$ is satisfied whenever the sample size $n$ obeys
\begin{multline*}
n < \frac{(1-\epsilon)}{12} \frac{\eta}{(1-\eta)^2} \frac{\rho_l}{C_{\max}} \\
\times \left( 1 + \frac{2\sqrt{\sigma^2 \log n}}{\lambda_{n,e} \sqrt{n}} \right)^{-2} k \log(n-s) \log(p-k).
\end{multline*}

\section{Conclusion}
\label{sec::conclusion}

In this paper, we studied the $\ell_1$-constrained minimization problem for sparse linear regression when the observations are grossly corrupted. We proposed the extended Lasso method which is a natural generalization of the Lasso for recovering both the regression and the error vector effectively. Our main contribution was to establish that this recovery is faithful, under both parameter estimation and variable selection criterions, even when the error magnitude is arbitrarily large and the fraction of error is close to unity. Specifically, our first result indicated that the $\ell_2$ estimation error is bounded via the introduction of the extended restricted eigenvalue (RE) condition evaluated on the combination matrix $[X \text{   } I]$. Our next results considered the exact signed support recovery for a class of random Gaussian design matrices. We showed that the sign consistency is indeed possible even when almost all the observations are significantly corrupted. More interestingly, we established the lower and upper bounds for the sample size such that the extended Lasso succeeds or fails in recovering the supports with high probability. This number of observations is scaled in term of the model dimension $p$, the sparsity index $k$, and the fraction error $\eta = s/n$. Notably, all of our results are consistent with that of the standard Lasso in the absence of sparse error.





There are a number of extensions and open questions related to this work. First, our setup can be extended to robust group/multivariate Lasso model. This model has been shown to outperform the conventional Lasso in many practical applications as well as theoretical analysis (e.g. \cite{OWJ_2011_J}, \cite{HZ_2010_J}, \cite{LPTG_2009_C}, \cite{NW_2011_J}). It would be interesting to obtain the upper and lower bound of the sample size when a significant fraction of observations is corrupted in this setting. Another interesting direction is to consider a more general situation where both the observations and the data matrix are corrupted/missing. In a recent paper, Loh and Wainwright \cite{LW_2011_C} established the consistency of the Lasso with noisy/corrupted/missing data matrix. Whether similar results would hold for more general setting is an interesting open problem. Lastly, although our current work focused exclusively on linear regression, it would be interesting to investigate the sparse additive models (e.g. \cite{RLLW_SpAM_2009_J}, \cite{MGB_2009_J}) under grossly corrupted observations.





\section{Appendix}
\label{sec::appendix}

\subsection{Proof of Lemma \ref{lem::bound inf norm of Xsct inv Xsct Xsct z}}
\label{app::proof bound inf norm of Xsct inv Xsct Xsct z}

Decomposing $X_{S^cT}$ as $X_{S^cT}= W_{S^cT}\Sigma_{TT}$ where $W_{S^cT} \in \R^{(n-s) \times k}$ is the random matrix with i.i.d. normal Gaussian entries, we have $X_{S^cT} (X^*_{S^c T} X_{S^cT})^{-1} = W_{S^cT} (W^*_{S^cT} W_{S^cT})^{-1} \Sigma_{TT}^{-1/2}$. Consider now the compact singular value decomposition of $W_{S^c T}$
$$
W_{S^c T} = U D V^*, \quad U \in \R^{(n-s)\times k} \text{ and } D, V \in \R^{k \times k}.
$$

\noindent Since $W_{S^c T}$ is a Gaussian random matrix with i.i.d. entries, columns of $U$ are orthogonal vectors selected uniformly at random. We can consider $U$ as a random matrix distributed on the Haar measure. We have
$$
X_{S^cT} (X^*_{S^c T} X_{S^cT})^{-1} z = U D^{\dag} V^* \Sigma_{TT}^{-1/2} z.
$$

\noindent Using the random matrix concentration inequality in (\ref{inq::spectral norm bound of X^* X}), we have with probability at least $1 - e^{-k}$
$$
\norm{W_{S^cT}} \leq \sqrt{n-s} \left( 1 + 4 \sqrt{\frac{k}{n-s}} \right)^{1/2} .
$$

\noindent In addition, from (\ref{inq::spectral norm bound of inverse of X^* X}), we have with high probability
$$
\norm{( W^*_{S^cT} W_{S^cT} )^{-1}} \leq \left(1 + 4 \sqrt{\frac{k}{n-s}} \right) \frac{1}{n-s}.
$$

\noindent Combining these pieces together, we conclude that
\begin{align*}
\norm{D^{\dag}} &= \norm{W_{S^cT} (W^*_{S^cT} W_{S^cT})^{-1}} \\
&\leq  \left( 1 + 4 \sqrt{\frac{k}{n-s}} \right)^{3/2} \frac{1}{\sqrt{n-s}} \leq \frac{\sqrt{1+\epsilon}}{\sqrt{n-s}},
\end{align*}
assuming that $k$ is sufficiently smaller than $(n-s)$.

Next, our goal is to bound
\begin{align*}
\norm{ U D^{\dag} V^* \Sigma^{-1/2}_{TT} z}_{\infty} &= \max_i |e_i^* U D^{\dag} V^* \Sigma^{-1/2}_{TT} z | \\
&= \max_i | \inner{U^*, D^{\dag} V^* \Sigma^{-1/2}_{TT} z e_i^*} | \\
&:= \max_i | f_i(U) | ,
\end{align*}
where $f_i$ is the function acting on the random matrix $U$, $f_i : \R^{|S^c| \times k} \rightarrow \R$.

First we show that $f_i(U)$ is Lipschitz (with respect to the Euclidean norm) with constant at most $\norm{f_i}_{L} = \sqrt{\frac{(1+ \epsilon)k}{C_{\min} (n-s)}} \norm{z}_{\infty}$. Indeed, for any given pair $U_1$ $U_2 \in \R^{|S^c| \times k}$, we have
\begin{align*}
| f_i( U_1) - f_i (U_2)  | &= |\inner{ U_1 - U_2,  D^{\dag} V^* \Sigma^{-1/2}_{TT} z e_i^*}| \\
&\leq \norm{U_1 - U_2}_F \norm{ D^{\dag} V^* \Sigma^{-1/2}_{TT} z e_i^*}_F \\
&\leq \norm{U_1 - U_2}_F \norm{D^{\dag} V^*} \norm{\Sigma^{-1/2}_{TT}} \norm{ z e_i^*}_F \\
&\leq \norm{U_1 - U_2}_F \frac{ \sqrt{1+\epsilon}}{\sqrt{n-s}} \frac{1}{\sqrt{C_{\min}}} \norm{z}_2 \\
&\leq  \sqrt{\frac{(1+\epsilon) k}{(n-s) C_{\min}}} \norm{U_1 - U_2}_F \norm{z}_{\infty}.
\end{align*}

\noindent Since the distribution of $U$ is invariant under the orthogonal transformation $U \mapsto -U $, $f(U)$ is a symmetric random variable and zero is a median. Hence, by the measure of concentration with respect to Haar measure in Lemma \ref{lem::haar measure}, we get
\begin{align*}
\Prob \left( f_i (U) \geq \tau \right) &\leq \exp \left( - \frac{\tau^2 (n-s) }{8 \norm{f_i}^2_{L}} \right) \\
&= \exp \left( - \frac{ C_{\min} (n-s)^2 \tau^2 }{8(1+\epsilon) k \norm{z}^2_{\infty}  } \right).
\end{align*}

\noindent Set $\tau := \frac{2\lambda}{3 \sqrt{n}} \left( 1 - \frac{2\sigma \sqrt{\log n}}{\lambda_{n,e} \sqrt{n}} \right) \norm{z}_{\infty}$ and take the union bound over all $i \in S^c$, we have
\begin{align*}
&\Prob \left( \norm{ U D^{\dag} V^* \Sigma^{-1/2}_{TT} z}_{\infty} \geq \frac{\lambda}{2} \norm{z}_{\infty} \right)  \\
&\leq (n-s) \exp \left( - \frac{ C_{\min} (n-s)^2 \lambda^2 }{12(1+\epsilon) n k  } \left( 1 - \frac{2\sigma \sqrt{\log n}}{\lambda_{n,e} \sqrt{n}} \right)^2\right).
\end{align*}

\noindent This probability vanishes at rate $\exp(- c \log n)$ provided that
$$
(n-s)^2 > 12 (1+\epsilon) \left( 1 - \frac{2\sigma \sqrt{\log n}}{\lambda_{n,e} \sqrt{n}} \right)^{-2} \frac{ nk \log n}{C_{\min} \lambda^2}.
$$

\noindent Replacing the expression of $\lambda$ in (\ref{eqt::equation of lambda}) and $s = \eta n$, the above condition is equivalent to
\begin{multline*}
\frac{n}{\log n} \geq C (1+\epsilon) \frac{\eta}{(1-\eta)^2} \frac{\max \{\rho_u, D^+_{\max} \}}{C_{\min} \gamma^2}  \\
\times \left( 1 - \frac{2\sigma \sqrt{\log n}}{\lambda_{n,e} \sqrt{n}} \right)^{-2} k \log(p-k),
\end{multline*}
where $C$ is a numerical constant smaller than $48$.


\subsection{Proof of Lemma \ref{lem::inachievability - main support lemma}}

Recall the decomposition of $X_{S^cT}$: $X_{S^cT} = W_{S^cT} \Sigma^{1/2}_{TT}$, we have
\begin{align*}
&X_{S^cT} (X^*_{S^cT} X_{S^cT})^{-1} z - \frac{1}{n-s} X_{S^c T} z \\
&= \left( W_{S^cT} (W^*_{S^cT} W_{S^cT})^{-1} - \frac{1}{n-s} W_{S^cT} \right) \Sigma^{-1/2} z.
\end{align*}

\noindent Notice that $W_{S^c T}$ is an $(n-s) \times k$ matrix with independent Gaussian entries with zero mean and unit variance. Consider now the reduced singular value decomposition of $W_{S^c T}$
$$
W_{S^cT} = U D V^*, \quad U \in \R^{(n-s)\times k} \text{ and } D, V \in \R^{k \times k}.
$$
Then the columns of $U$ are $k$ orthonormal vectors selected uniformly at random. We can think of $U$ as a random matrix distributed on the Haar measure. The above equation is now formulated as
\begin{align*}
\frac{1}{n-s} U D \left[ \left( \frac{D^* D}{n-s} \right)^{-1} - I \right] V \Sigma^{-1/2} z =:  U \widetilde{D} V \Sigma^{-1/2} z .
\end{align*}

\noindent It is clear that $\norm{\widetilde{D}} \leq \frac{1}{n-s} \norm{W_{S^cT}} \norm{( \frac{W_{S^cT}^*W_{S^cT}}{n-s})^{-1} - I}$. Recalling the random matrix concentration bounds (\ref{inq::spectral norm bound of X^* X}) and (\ref{inq::spectral norm bound of inverse of X^* X}), we have $\norm{\frac{W_{S^cT}}{\sqrt{n-s}}} \leq  (1 + 4\sqrt{\frac{k}{n-s}})^{1/2}$. Therefore,
\begin{equation*}
\norm{\widetilde{D}} \leq \frac{4 \sqrt{k}}{n-s} \left( 1 + 4\sqrt{\frac{k}{n-s}} \right)^{1/2} =: (1 + \epsilon) \frac{4\sqrt{k}}{n-s},
\end{equation*}
where we choose $\epsilon \geq 4\sqrt{k/(n-s)}$.

Our goal now is to establish an upper bound of $\norm{U \widetilde{D} V \Sigma^{-1/2} z}_{\infty}$, which can be rewritten as
\begin{align*}
\max_i | e^*_i U \widetilde{D} V \Sigma^{-1/2} z | &= \max_i | \inner{U, \widetilde{D} V \Sigma^{-1/2} z e^*_i } | \\
&:= \max_i f_i (U),
\end{align*}
where $f_i$ is a function operating on the random matrix $U$, $f_i : \R^{(n-s) \times k} \mapsto \R$.

First we show that $f_i(U)$ is Lipschitz (with respect to the Euclidean norm) with constant at most $\norm{f_i}_{L} = \frac{ 4 (1+\epsilon)\sqrt{ k}}{n-s} \frac{1}{\sqrt{C_{\min}}} \norm{z}_2$. Indeed, for any given pair $U_1$ $U_2 \in \R^{|S^c| \times k}$, we have
\begin{align*}
| f_i( U_1) - f_i (U_2)  | &= |\inner{ U_1 - U_2,  \widetilde{D} V^* \Sigma^{-1/2}_{TT} z e_i^*}| \\
&\leq \norm{U_1 - U_2}_F \norm{ \widetilde{D} V^* \Sigma^{-1/2}_{TT} z e_i^*}_F \\
&\leq \norm{U_1 - U_2}_F \norm{\widetilde{D} V^*} \norm{\Sigma^{-1/2}_{TT}} \norm{ z e_i^*}_F \\
&\leq \norm{U_1 - U_2}_F \frac{ 4 (1+\epsilon)\sqrt{ k}}{n-s} \frac{1}{\sqrt{C_{\min}}} \norm{z}_2. \\
%
\end{align*}

\noindent Since the distribution of $U$ is invariant under the orthogonal transformation $U \mapsto -U $, $f(U)$ is a symmetric random variable and zero is a median. Hence, by the measure of concentration with respect to Haar measure (Lemma \ref{lem::haar measure}), we get
\begin{align*}
\Prob \left( f_i (U) \geq \tau \right) &\leq \exp \left( - \frac{\tau^2 (n-s) }{8 \norm{f_i}^2_{L}} \right) \\
&= \exp \left( - \frac{ C_{\min} (n-s)^3 \tau^2 }{128 (1+\epsilon)^2 k \norm{z}^2_2  } \right).
\end{align*}

\noindent Setting $\tau^2 : = \frac{256(1+\epsilon)^2 \norm{z}_2^2 k \log(n-s)}{C_{\min} (n-s)^3}$ and taking the union bound over all $i \in S^c$, we have
\begin{multline*}
\Prob \left( \norm{ U D^{\dag} V^* \Sigma^{-1/2}_{TT} z}_{\infty} \geq \frac{16 (1+\epsilon) \norm{z}_2 \sqrt{ k \log(n-s)} }{\sqrt{C_{\min} (n-s)^3}} \right)  \\
\leq \exp(-\log(n-s)),
\end{multline*}
as claimed.

\subsection{Proof of Lemma \ref{lem::lower l-infty bound of Xsct z}}

We have $X_{S^c T} = W_{S^c T} \Sigma^{1/2}_{TT}$, where $W_{S^c T}$ is a standard Gaussian matrix of size $(n-s) \times k$. Thus, $X_{S^c  T} \Sigma^{-1}_{TT} z = W_{S^c T} \Sigma^{-1/2}_{TT} z$, which leads to
\begin{align*}
\norm{X_{S^c  T} \Sigma^{-1}_{TT} z}_{\infty} &= \max_{i \in S^c} |\inner{e_i, W_{S^c T} \Sigma^{-1/2}_{TT} z}| \\
&=: \max_{i \in S^c} |f_i (W_{S^cT})|,
\end{align*}
where $e_i \in \R^{(n-s)}$ is the standard vector whose entry at $i$-th location receive unit value and zeros elsewhere. In order to lower the bound of the random variable $\max_i f_i (W_{S^cT})$, the first step is to show that it is sharply concentrated around its expectation.

\begin{lem}
\label{lem::concentration bound of max f(wi)}
For any $\tau > 0$, we have
\begin{multline}
\Prob \left( | \max_i f_i (W_{S^cT}) - \E \max_i f_i(W_{S^cT}) | \geq \tau \right) \\
\leq 4 \exp \left( - \frac{\tau^2}{2 \norm{ \Sigma_{TT}^{-1/2} z}_2^2} \right).
\end{multline}
\end{lem}

\noindent Select $\tau := \norm{ \Sigma_{TT}^{-1/2} z}_2 \sqrt{\frac{1}{2}\log (n-s)}$, we conclude that with probability greater than $1 - 4 \exp(- \frac{1}{4}\log (n-s))$
\begin{equation}
\label{inq::bound maximum f_i of W_S^CT}
\max_i f_i (W_{S^cT}) \geq \E \max_i f_i(W_{S^cT}) - \tau . 
\end{equation}

At the second step, we need to lower the bound $\E \max_i f_i (W_{S^cT})$. This can be estimated via Sudakov-Fernique inequality \cite{LT_1991_B}. We have,
\begin{equation*}
\begin{split}
\E (f_i(W_{S^cT}) - f_j(W_{S^cT}))^2 &= 2 z^* \Sigma^{-1}_{TT} z = 2 \norm{\Sigma^{-1/2}_{TT} z}_2^2.
%
\end{split}
\end{equation*}

\noindent Consequently, if we denote $g_i$, $1 \leq i \leq (n-s)$ as a sequence of $\oper N(0, \norm{\Sigma^{-1/2}_{TT} z}_2^2)$ Gaussian random variables, then we have established a lower bound
$$
\E (f_i(W_{S^cT}) - f_j(W_{S^cT}))^2 \geq \E (g_i - g_j)^2
$$

\noindent Therefore, the Sudakov-Fernique inequality \cite{LT_1991_B} suggests that the maximum over $f(w_i)$ dominates the maximum over $g_i$. In particular, we have $\E \max_i f_i(W_{S^cT}) \geq \E \max_i g_i$. Moreover, since $\{ g_i \}$ are i.i.d. random variables, by the standard bound for Gaussian extreme, for all $\delta > 0$, we have
\begin{multline*}
\E \max_i f(W_{S^cT}) \geq \E \max_i g_i \\
\geq \norm{\Sigma^{-1/2}_{TT} z}_2 \sqrt{(2 - \delta)\log(n-s)}.
\end{multline*}

\noindent Substituting this expectation bound into (\ref{inq::bound maximum f_i of W_S^CT}) yields
\begin{align*}
\max_i f_i (W_{S^cT}) &\geq (\sqrt{2-\delta} - \sqrt{1/2}) \norm{\Sigma^{-1/2}_{TT} z}_2 \sqrt{\log(n-s)} \\
&> \frac{2}{3} \norm{\Sigma^{-1/2}_{TT} z}_2 \sqrt{\log(n-s)}
\end{align*}
for $\delta$ arbitrarily close to zero. Furthermore, using the standard bound $\norm{\Sigma^{-1/2}_{TT} z}_2 \geq \frac{\norm{z}_2}{\norm{\Sigma^{1/2}_{TT}}_2} \geq  \frac{\norm{z}_2}{\sqrt{C_{\max}}}$, we complete the proof.

\begin{proof} [Proof of Lemma \ref{lem::concentration bound of max f(wi)}]

By the standard Gaussian concentration theorems \cite{LT_1991_B}, let $w$ be a standard Gaussian measure on $\R^n$ and $f$ be a Lipschitz function with Lipschitz constant $\norm{f}_{lip}$. Then,
\begin{equation}
\label{eqt::function tail bound on standard Gaussian concentration}
\Prob ( f(w) - \E f(w) \geq \tau) \leq 4 \exp(- \tau^2/ 2 \norm{f}^2_{lip}).
\end{equation}

\noindent We now consider the function $f(W_{S^cT}) := \max_i f_i (W_{S^cT})$ operating on the standard Gaussian matrix $W_{S^c T}$. We have
\begin{align*}
f(W_{S^cT}^1) - f(W_{S^cT}^2)  &= \max_i \inner{e_i,  W^1_{S^cT}  \Sigma_{TT}^{-1/2} z } \\
&{ }- \max_k \inner{ e_k,  W^2_{S^cT}  \Sigma_{TT}^{-1/2} z } \\
&\leq  \max_i \inner{ e_i, (W^1_{ST} - W^2_{ST}) \Sigma_{TT}^{-1/2} z } \\
&\leq \norm{ \Sigma_{TT}^{-1/2} z}_2 \norm{W^1_{ST} - W^2_{ST}}_F \\
\end{align*}
where the second inequality follows from the Cauchy-Schwartz inequality. Applying (\ref{eqt::function tail bound on standard Gaussian concentration}) with Lipschitz constant $\norm{ \Sigma_{TT}^{-1/2} z}_2$ completes our proof.
\end{proof}

\subsection{Some concentration inequalities }
\label{appen::concentration inequalities}

In this section, we restate some well-known large deviation bounds for ease of reference. The first is a bound of sum of Gaussian random variables.

\begin{lem}
\label{lem::bound for a sum of sub-Gaussian random variables}
Let $Z_1, ..., Z_n$ be independent and zero-mean Gaussian random variables with parameters $\sigma_1^2, ..., \sigma_n^2$. Then
$$
\Prob \left( |\sum_{i=1}^n Z_i | \geq \tau \right) \leq 2 \exp \left( - \frac{\tau^2}{2 \sum_{i=1}^n \sigma_i^2} \right).
$$
\end{lem}

This bound comes directly from a standard Gaussian bound. For a Gaussian variable $Z \sim \oper N(0,\sigma^2)$, we have with all $\tau > 0$
\begin{equation}
\label{lem::bound for a sub-Gaussian random variables}
\Prob (|Z| \geq \tau) \leq 2 \exp \left( - \frac{\tau^2}{2 \sigma^2} \right).
\end{equation}

The following tail bounds on the Chi-square variates taken from \cite{LM_1998_J} are useful
\begin{lem}
\label{lem::chi square bound}
Let $X$ be a centralized $\chi^2$-variate with $d$ degree of freedom. Then for all $\tau \in (0, 1/2)$, we have
\begin{align*}
    &\Prob \left( X \geq d (1+\tau) \right) \leq \exp \left( - \frac{3}{16} d \tau^2 \right) \\
    &\Prob \left( X \leq d (1-\tau) \right) \leq \exp \left( - \frac{1}{4} d \tau^2 \right).
\end{align*}
\end{lem}

We also recall some well-known concentration inequalities from random matrix theory 
\begin{lem}
\label{lem::standard random matrix inequality of singular values}
Let $X^{n \times k}$ be a random matrix, whose entries are standard Gaussian random variables. Denote by $\sigma_{\min}$ and $\sigma_{\max}$ the smallest and largest singular values of $X$. Then we have
\begin{align*}
    &\Prob \left( 1 - \sigma_{\min}(X)/\sqrt{n} \geq \sqrt{\frac{k}{n}} + \tau \right) \leq \exp \left( -n\ \tau^2/2 \right)  \\
\label{inq::biggest singular of Gaussian matrix}
    &\Prob \left( \sigma_{\max}(X)/\sqrt{n} - 1 \geq \sqrt{\frac{k}{n}} + \tau \right) \leq \exp \left( -n \tau^2/2 \right).
\end{align*}
\end{lem}

By setting $\tau = \sqrt{\frac{k}{n}}$, we conclude that with probability at least $1 - \exp(-k/2)$,
\begin{equation}
\label{inq::bound singular values of matrix X^*X}
\begin{split}
( 1 - 2 \sqrt{k/n} )^2 &\leq \sigma_{\min} \left( X^* X/n \right) \\
&\leq \sigma_{\max} \left( X^* X/n \right)\leq ( 2 \sqrt{k/n} + 1 )^2.
\end{split}
\end{equation}

A consequence of this quantity is another singular value bound for the inverse matrix of $X^* X$. We have with probability greater than $1 - \exp(-k/2)$,
\begin{equation}
\label{inq::bound singular values of inverse matrix of X^*X}
\begin{split}
\frac{1}{( 2 \sqrt{k/n} + 1 )^2} &\leq \sigma_{\min} \left( (X^*X/n)^{-1} \right) \\
&\leq \sigma_{\max} \left( (X^* X/n)^{-1} \right)\leq \frac{1}{( 1 - 2 \sqrt{k/n} )^2}.
\end{split}
\end{equation}

From the above two set of inequality and assumption that $k \leq n$, we conclude that with probability greater than $1 - \exp(-k/2)$,
\begin{eqnarray}
\label{inq::spectral norm bound of X^* X}
  \norm{\frac{X^* X}{n} - I} &\leq& 4 \sqrt{\frac{k}{n}} \\
\label{inq::spectral norm bound of inverse of X^* X}
  \norm{\left( \frac{X^* X}{n} \right)^{-1} - I} &\leq& 4 \sqrt{\frac{k}{n}} .
\end{eqnarray}

For random matrices whose rows are i.i.d and have distribution $\oper N(0, \Sigma)$, we can achieve a similar spectral norm bound. We have with probability at least $1 - \exp(-k/2)$
\begin{eqnarray}
\label{inq::spectral norm bound of general X^* X}
  \norm{\frac{X^* X}{n} - \Sigma} &\leq& 4 \sigma_{\max} (\Sigma) \sqrt{\frac{k}{n}} \\
\label{inq::spectral norm bound of general inverse matrix X^* X}
  \norm{\left( \frac{X^* X}{n} \right)^{-1} - \Sigma^{-1}} &\leq& \frac{4}{\sigma_{\min}(\Sigma)} \sqrt{\frac{k}{n}}.
\end{eqnarray}

Finally, the following lemma states an useful concentration inequality on Haar measure \cite{Ledoux_2001_B}.
\begin{lem}
\label{lem::haar measure}
Support $k < n$ and let $f: \R^{n \times k} \mapsto R$ with Lipschitz norm
$$
\norm{f}_{\text{L}} = \sup_{X \neq Y} \frac{f(X)-f(Y)}{X-Y}.
$$
Then if $U$ is distributed according to the Haar measure,
$$
\Prob (f(U) \geq \text{median} (f) + \tau) \leq \exp(- \frac{m \tau^2}{8 \norm{f}_{\text{L}}^2 }).
$$
\end{lem}

\bibliographystyle{IEEEtran}
\bibliography{IEEEabrv,all_references}

\begin{thebibliography}{10}
\providecommand{\url}[1]{#1}
\csname url@samestyle\endcsname
\providecommand{\newblock}{\relax}
\providecommand{\bibinfo}[2]{#2}
\providecommand{\BIBentrySTDinterwordspacing}{\spaceskip=0pt\relax}
\providecommand{\BIBentryALTinterwordstretchfactor}{4}
\providecommand{\BIBentryALTinterwordspacing}{\spaceskip=\fontdimen2\font plus
\BIBentryALTinterwordstretchfactor\fontdimen3\font minus
  \fontdimen4\font\relax}
\providecommand{\BIBforeignlanguage}[2]{{%
\expandafter\ifx\csname l@#1\endcsname\relax
\typeout{** WARNING: IEEEtran.bst: No hyphenation pattern has been}%
\typeout{** loaded for the language `#1'. Using the pattern for}%
\typeout{** the default language instead.}%
\else
\language=\csname l@#1\endcsname
\fi
#2}}
\providecommand{\BIBdecl}{\relax}
\BIBdecl

\bibitem{Tibshirani_Lasso_1996_J}
R.~Tibshirani, ``Regression shrinkage and selection via the lasso,'' \emph{J.
  Roy. Statist. Soc. Ser. B}, vol.~58, no.~1, pp. 267--288, 1996.

\bibitem{CT_Dantzig_2007_J}
E.~J. Cand\`{e}s and T.~Tao, ``The {D}antzig selector: statistical estimation
  when p is much larger than n,'' \emph{Ann. Statist.}, vol.~35, no.~6, pp.
  2313--2351, 2007.

\bibitem{YL_2006_J}
M.~Yuan and Y.~Lin, ``Model selection and estimation in regression with grouped
  variables,'' \emph{J. Roy. Statist. Soc. Ser. B}, vol.~68, no.~1, pp. 49--67,
  2006.

\bibitem{ZRY_2009_J}
P.~Zhao, G.~Rocha, and B.~Yu, ``The composite absolute penalties family for
  grouped and hierarchical variable selection,'' \emph{Ann. Statist.}, vol.~37,
  pp. 3469--3497, 2009.

\bibitem{ZY_Lasso_2006_J}
P.~Zhao and B.~Yu, ``On model selection consistency of {L}asso,'' \emph{J.
  Machine Learn. Res.}, vol.~7, pp. 2541--2563, 2006.

\bibitem{MY_Lasso_2009_J}
N.~Meinshausen and B.~Yu, ``Lasso-type recovery of sparse representations for
  high-dimensional data,'' \emph{Ann. Statist.}, vol.~37, no.~1, pp.
  2246--2270, 2009.

\bibitem{MB_Lasso_2008_J}
N.~Meinshausen and P.~B$\ddot{\text{u}}$hlmann, ``High dimensional graphs and
  variable selection with the lasso,'' \emph{Ann. Statist.}, vol.~34, no.~3,
  pp. 1436--1462, 2008.

\bibitem{Wainwright_Lasso_2009_J}
M.~J. Wainwright, ``Sharp thresholds for high-dimensional and noisy sparsity
  recovery using l1-constrained quadratic programming (lasso ),'' \emph{IEEE
  Trans. Inf. Theory}, vol.~55, no.~5, pp. 2183--2202, May 2009.

\bibitem{CP_Lasso_2009_J}
E.~J. Cand\`es and Y.~Plan, ``Near-ideal model selection by l1 minimization,''
  \emph{Ann. Statist.}, vol.~37, pp. 2145--2177, 2009.

\bibitem{BRT_Lasso_2009_J}
P.~Bickel, Y.~Ritov, and A.~Tsybakov, ``Simultaneous analysis of {L}asso and
  {D}antzig selector,'' \emph{Ann. Statist.}, vol.~37, no.~4, pp. 1705--1732,
  2009.

\bibitem{Zhang_2009_J}
T.~Zhang, ``Some sharp performance bounds for least squares regression with l1
  regularization,'' \emph{Ann. Statist.}, vol.~37, no.~5, pp. 2109--2144, 2009.

\bibitem{BTW_2007_J}
F.~Bunea, A.~Tsybakov, and M.~Wegkamp, ``Sparsity oracle inequalities for the
  lasso,'' \emph{Elec. Journal Statist.}, vol.~1, pp. 169--194, 2007.

\bibitem{Bunea_2008}
F.~Bunea, ``Honest variable selection in linear and logistic regression models
  via $\ell_1$ and $\ell_1 + \ell_2$ penalization,'' \emph{Elec. Journal
  Statist.}, vol.~2, pp. 1153--1194, 2008.

\bibitem{OWJ_2011_J}
G.~Obozinski, M.~J. Wainwright, and M.~I. Jordan, ``Support union recovery in
  high-dimensional multivariate regression,'' \emph{Ann. Statist.}, vol.~39,
  no.~1, pp. 1--47, 2011.

\bibitem{HZ_2010_J}
J.~Huang and T.~Zhang, ``The benefit of group sparsity,'' \emph{Ann. Statist.},
  vol.~38, no.~4, pp. 1978--2004, 2010.

\bibitem{Tropp_Relax_2006_J}
J.~A. Tropp, ``Just relax: Convex programming methods for identifying sparse
  signals,'' \emph{IEEE Trans. Inf. Theory}, vol.~51, no.~3, pp. 1030--1051,
  Mar. 2006.

\bibitem{WYGSM_Face_2009_J}
J.~Wright, A.~Y. Yang, A.~Ganesh, S.~S. Sastry, and Y.~Ma, ``Robust face
  recognition via sparse representation,'' \emph{IEEE Trans. Pattern Anal.
  Mach. Intell.}, vol.~31, no.~2, pp. 210--227, Feb. 2009.

\bibitem{EV_2009_C}
E.~Elhamifar and R.~Vidal, ``Sparse subspace clustering,'' in \emph{IEEE Conf.
  Comput. Vis. Patt. Recog. (CVPR)}, Miami Beach, FL, USA, June 2009, pp.
  2790--2797.

\bibitem{LDB_corruption_2009_C}
J.~N. Laska, M.~A. Davenport, and R.~G. Baraniuk, ``Exact signal recovery from
  sparsely corrupted measurements through the pursuit of justice,'' in
  \emph{Asilomar conf. Sig. Sys. Comput.}, Pacific Grove, CA, USA, Nov. 2009,
  pp. 1556--1560.

\bibitem{WM_denseError_2010_J}
J.~Wright and Y.~Ma, ``Dense error correction via l1 minimization,'' \emph{IEEE
  Trans. Inf. Theory}, vol.~56, no.~7, pp. 3540--3560, July 2010.

\bibitem{LWW_grossError_2010_C}
Z.~Li, F.~Wu, and J.~Wright, ``On the systematic measurement matrix for
  compressed sensing in the presence of gross error,'' in \emph{Data compress.
  conf. (DCC)}, Snowbird, UT, USA, Mar. 2010, pp. 356--365.

\bibitem{NT_GrossError_2010_J}
N.~H. Nguyen and T.~D. Tran, ``Exact recoverability from dense corrupted
  observations via $l_1$ minimization,'' Feb. 2011, preprint at
  \url{http://arxiv.org/abs/1102.1227}.

\bibitem{Li_2011_J}
X.~Li, ``Compressed sensing and matrix completion with constant proportion of
  corruptions,'' April 2011, preprint at \url{http://arxiv.org/abs/1104.1041}.

\bibitem{CLMW_RobustPCA_2009_J}
E.~J. Cand\`es, X.~Li, Y.~Ma, and J.~Wright, ``Robust principal component
  analysis?'' \emph{Journal of the ACM}, vol.~58, no.~3, pp. 1--37, May 2011.

\bibitem{XCS_RPCA_2010_C}
H.~Xu, C.~Caramanis, and S.~Sanghavi, ``Robust {PCA} via outlier pursuit,'' in
  \emph{Ad. Neural Infor. Proc. Sys. (NIPS)}, Vancouver, BC, Canada, Dec. 2010,
  pp. 2496--2504.

\bibitem{ANW_RPCA_2011_C}
A.~Agarwal, S.~Negahban, and M.~Wainwright, ``Noisy matrix decomposition via
  convex relaxation: Optimal rates in high dimensions,'' in \emph{Proc. 28th
  Inter. Conf. Mach. Learn. (ICML)}, Bellevue, Washington, USA, June 2011, pp.
  1129--1136.

\bibitem{HBRN_2008_J}
J.~Haupt, W.~Bajwa, M.~Rabbat, and R.~Nowak, ``Compressed sensing for networked
  data,'' \emph{IEEE Signal Process. Mag.}, vol.~25, no.~2, pp. 92--101, Mar.
  2008.

\bibitem{LMJ_2012_J}
Y.~Lee, S.~N. MacEachern, and Y.~Jung, ``Regularization of case-specific
  parameters for robustness and efficiency,'' \emph{Statis. Science}, 2012, to
  appear.

\bibitem{WLJ_2007_J}
H.~Wang, G.~Li, and G.~Jiang, ``Robust regression shrinkage and consistent
  variable selection through the {LAD}-{L}asso,'' \emph{Journal Busi. Econ.
  Statist.}, vol.~25, no.~3, pp. 347--355, July 2007.

\bibitem{CRT_CS_2004_J}
E.~J. Cand\`{e}s, J.Romberg, and T.~Tao, ``Robust uncertainty principles: exact
  signal reconstruction from highly incomplete frequency information,''
  \emph{IEEE Trans. Inf. Theory}, vol.~52, no.~2, pp. 5406--5425, Feb. 2006.

\bibitem{SKPB_2011_C}
C.~Studer, P.~Kuppinger, G.~Pope, and H.~Bolcskei, ``Sparse signal recovery
  from sparsely corrupted measurements,'' in \emph{Proc. Inter. Symp. Inf.
  Theory (ISIT)}, St. Pertersburg, Russia, Aug. 2011, pp. 1422--1426.

\bibitem{SB_2011_J}
C.~Studer and R.~G. Baraniuk, ``Stable restoration and separation of
  approximately sparse signals,'' July 2011, submitted to Applied Comput. Har.
  Anal., Preprint at \url{http://arxiv.org/abs/1107.0420}.

\bibitem{RWY_RE_2010_J}
G.~Raskutti, M.~J. Wainwright, and B.~Yu, ``Restricted eigenvalue properties
  for correlated gaussian designs,'' \emph{J. Machine Learn. Res.}, vol.~11,
  pp. 2241--2259, 2010.

\bibitem{NRWY_2010_J}
S.~Negahban, P.~Ravikumar, M.~J. Wainwright, and B.~Yu, ``A unified framework
  for high-dimensional analysis of {M}-estimators with decomposable
  regularizers,'' in \emph{Ad. Neural Infor. Proc. Sys. (NIPS)}, Vancouver, BC,
  Canada, Dec. 2009.

\bibitem{GB_Lasso_2009_J}
S.~van~de Geer and P.~B$\ddot{\text{u}}$hlmann, ``On the conditions used to
  prove oracle results for the lasso,'' \emph{Elec. J. Statist.}, vol.~3, no.
  1360-1392, 2009.

\bibitem{CRT_Stability_2006a_J}
E.~J. Cand\`{e}s, J.Romberg, and T.~Tao, ``Stable signal recovery from
  incomplete and inaccurate measurements,'' \emph{Comm. Pure Appl. Math.},
  vol.~59, no.~8, pp. 1207--1223, Aug. 2006.

\bibitem{LT_1991_B}
M.~Ledoux and M.~Talagrand, \emph{Probability in Banach Space: Isoperimetry and
  Processes}.\hskip 1em plus 0.5em minus 0.4em\relax Springer, 1991.

\bibitem{LPTG_2009_C}
K.~Lounici, M.~Pontil, A.~Tsybakov, and S.~van~de Geer, ``Taking advantage of
  sparsity in multi-task learning,'' in \emph{Proc. Ann. Conf. Learn. Theory},
  Montreal, Canada, June 2009, pp. 73--82.

\bibitem{NW_2011_J}
S.~Negahban and M.~J. Wainwright, ``Simultaneous support recovery in high
  dimensions: Benefits and perils of block
  $\ell_1/\ell_\infty$-regularization,'' \emph{IEEE Trans. Inf. Theory},
  vol.~57, no.~6, pp. 3841--3863, June 2011.

\bibitem{LW_2011_C}
P.-L. Loh and M.~J. Wainwright, ``High-dimensional regression with noisy and
  missing data: Provable guarantees with non-convexity,'' in \emph{Ad. Neural
  Infor. Proc. Sys. (NIPS)}, Granada, Spain, Dec. 2011.

\bibitem{RLLW_SpAM_2009_J}
R.~Ravikumar, J.~Lafferty, H.~Liu, and L.~Wasserman, ``Sparse additive
  models,'' \emph{J. Royal Statist. Soc.: Series B}, vol.~71, no.~5, pp.
  1009--1030, Nov. 2009.

\bibitem{MGB_2009_J}
L.~Meier, S.~van~de Geer, and P.~B$\ddot{\text{u}}$hlmann, ``High-dimensional
  additive modeling,'' \emph{Ann. Statist.}, vol.~37, no.~6B, pp. 3779--3821,
  2009.

\bibitem{LM_1998_J}
B.~Laurent and P.~Massart, ``Adaptive estimation of a quadratic functional by
  model selection,'' \emph{Ann. Statist.}, vol.~28, no.~5, pp. 1303--1338,
  1998.

\bibitem{Ledoux_2001_B}
M.~Ledoux, \emph{The Concentration of Measure Phenomenon}.\hskip 1em plus 0.5em
  minus 0.4em\relax American Math. Soc., 2001.

\end{thebibliography}

\end{document}